\documentclass[final]{siamltex}

\usepackage[latin1]{inputenc}
\usepackage{algorithmic}
\usepackage{algorithm}
\usepackage{url}
\usepackage{graphics}
\usepackage{color}

\usepackage{graphicx}
\usepackage{psfrag}
\usepackage{amsmath}
\usepackage{tabularx}
\usepackage{verbatim}
\usepackage{float}
\usepackage{amssymb}
\usepackage{latexsym}
\usepackage{pifont}
\usepackage{subfigure}
\usepackage{ tikz}
\usetikzlibrary{arrows}
\usetikzlibrary{fit,positioning}

\newcommand{\KK}{\mathcal{K}}

\newcommand{\RR}{\mathbb{R}}

\newcommand{\CC}{\mathbb{C}}

\newcommand{\NN}{\mathbb{N}}

\newcommand{\kron}{\otimes}

\newcommand{\slambda}{s}

\begin{document}
\title{Computing delay Lyapunov matrices and $\mathcal{H}_2$ norms for 
	large-scale problems}

\author{Wim Michiels\thanks{KU Leuven, Department of Computer Science,
	Heverlee, Belgium,
	\texttt{\{Wim.Michiels@cs.kuleuven.be}}	%
\and%
Bin Zhou\thanks{Harbin Institute of Technology, Center for Control Theory and Guidance Technology, Harbin, China
	\texttt{\{binzhou@hit.edu.cn}}	
}

\maketitle

\begin{abstract}
A delay Lyapunov matrix corresponding to an exponentially stable system of linear time-invariant delay differential equations can be characterized as the solution of a boundary value problem  involving a matrix valued  delay differential equation. This boundary value problem can be seen as a natural generalization of the classical Lyapunov matrix equation. Lyapunov matrices play an important role in constructing Lyapunov functionals and in $\mathcal{H}_2$ optimal control.   
In this paper we present a general approach for computing delay Lyapunov matrices and $\mathcal{H}_2$ norms for systems with multiple discrete delays, whose applicability extends towards problems where the matrices are large and sparse, and the associated positive semidefinite matrix (the ``right-hand side'' for the standard Lyapunov equation), has a low rank.   
The problems addressed are challenging, because besides that the boundary value problem is matrix valued with a structure that much harder to exploit than in the delay-free case, its solution is in the generic situation non-smooth.     
In contract to existing methods that are based on solving the boundary value problem directly, our method is grounded in solving standard Lyapunov equations of increased dimensions.  It combines several ingredients: i) a spectral discretization of the system of delay equations, ii) a targeted similarity transformation which induces a desired structure and sparsity pattern and, at the same time, favors accurate low rank  solutions of the corresponding Lyapunov equation, and iii) a Krylov method for large-scale matrix Lyapunov equations.  The structure of the problem is exploited in such a way that the final algorithm does not involve a preliminary discretization step, and provides a fully dynamic construction of approximations of increasing rank. Interpretations in terms of a projection method directly applied to a standard linear infinite-dimensional system equivalent to the original time-delay  system are also given. Throughout the paper two didactic examples are used to illustrate the properties of the problem, the challenges and methodological choices, while numerical experiments are presented at the end to illustrate the effectiveness of the algorithm.

\end{abstract}


\section{Introduction}\label{sect:intro}

We consider a linear  system with multiple discrete delays,
\begin{equation}\label{sys-z}
\dot x(t)=A_0 x(t)+\sum_{i=1}^m A_i x(t-\tau_i),
\end{equation}
where $x(t)\in\RR^n$ is the state variable at time $t$, $A_i\in\mathbb{R}^{n\times n} are the system matrices$  and $\tau_i,\ i=1,\ldots,m$, represent time-delays, ordered such that
\[
0<\tau_1<\cdots<\tau_m.
\]
Throughout the paper we assume that the zero solution of (\ref{sys-z}) is exponentially stable, or equivalently, that all its characteristic roots, i.e., the solutions of equation
\[
\det\left(\lambda I-A_0-\sum_{i=1}^m A_i e^{-\lambda\tau_i}\right)=0,
\]
are confined to the open left half plane \cite{bookwim,bookniculescu}.
The fundamental solution  of (\ref{sys-z}), which we denote by $K:\RR \to\RR^{n\times n}$, is defined as the function satisfying
\begin{equation}\label{defK}
\left\{\begin{array}{ll}
\dot K(t)=A_0 K(t)+\sum_{i=1}^m A_i K(t-\tau_i),\ \  & \mathrm{for\ almost\ all\ } t\geq 0,\\
K(0)=I,  &
\\
K(t)=0,\ & \mathrm{for}\ t<0.
\end{array}\right.
\end{equation}
%
 The delay Lyapunov matrix for (\ref{sys-z}), associated with a positive semidefinite matrix, whose rank revealing decomposition reads as $B B^T$, 
where $B\in\mathbb{R}^{n\times r}$ is of full rank $r$, 
is defined as a function $P:\ \RR\to\RR^{n\times n}$ such that
\begin{equation}\label{defU}
P(t)=\int_0^{\infty} K(s) B B^T K^T(s+t) ds.
\end{equation}
Following from the exponential stability condition of (\ref{sys-z}), the delay Lyapunov matrix can be characterized as the \emph{unique} solution of the matrix valued ``boundary'' value problem
\begin{equation} \label{BVP}
\left\{\begin{array}{rcl}
\dot P(t)&=&P(t)A_0^T+\sum_{k=1}^m P(t-\tau_k)A_k^T,\ \ t\geq 0,  \\
P(-t)&=&P^T(t), \\
%
{-B B^T}&{=}&{P(0)A_0^T+A_0 P(0)} + \sum_{k=1}^m\left(P(-\tau_k)A_k^T+A_kP(\tau_k)\right),
\end{array}\right.
\end{equation}
see \cite{kharmatrix}.
There is also a dual formulation: with a positive semi-definite $(n\times n)$-matrix $C^T C $ we can associate Lyapunov matrix
\[
Q(t)=\int_0^{\infty} K^T(s) C^T C K(s+t)ds,
\]
which corresponds to the unique solution of  
\begin{equation} \label{BVPdual}
\left\{\begin{array}{rcl}
\dot Q(t)&=&  Q(t) A_0+\sum_{k=1}^m  Q(t-\tau_k) A_k,\ \ t\geq 0,  \\
Q(-t)&=&Q^T(t), \\
%
{-C^TC}&{=}&{Q(0)A_0+A_0^T Q(0)} + \sum_{k=1}^m\left(Q(-\tau_k)A_k+A_k^T Q(\tau_k)\right).
\end{array}\right.
\end{equation}
 Note that in the delay-free case the third equation in (\ref{BVP}) and the third one in (\ref{BVPdual}) reduce to a standard pair of primal and dual Lyapunov matrix equations.

The delay Lyapunov matrix is a building block for the construction of Lyapunov functionals of 
complete type, which are associated with necessary and sufficient stability conditions, see \cite{kharmatrix} for an excellent review. It should be remarked that in the literature on complete type Lyapunov functionals, the Lyapunov matrix is usually denoted by $U(t)$, which corresponds to $Q(t)$ in our adopted notation. Another comment is that in several works the Lyapunov matrix is alternatively defined directly as the solution of boundary value problem (\ref{BVP}) or (\ref{BVPdual}). In this way it can also be defined for an exponentially unstable system (provided the delay systems has no pair of eigenvalues $(\lambda_1,\lambda_2)$ such that $\lambda_1+\lambda_2=0$, see \cite{Kharitonov:2006:LYAPUNOVPLISCHKE,kharmatrix}), at the price that the aforementioned connection with the fundamental solution is lost.
 The Lyapunov matrix also plays a major role in the characterization of the $\mathcal{H}_2$ norm for system 
\begin{equation}\label{sys}
\left\{\begin{array}{lll}
\dot x(t)&=&A_0 x(t)+\sum_{i=1}^m A_i x(t-\tau_i)+B u(t), \\
y(t)&=&C x(t),
\end{array}\right.
\end{equation}
where  $u\in\CC^r$ is the input, $y\in\CC^{s}$ is the output and $B\in\mathbb{R}^{n\times r}$, respectively $C\in\mathbb{R}^{s\times n}$ are the input, respectively and out matrix of the model.  The transfer function of the system (\ref{sys}) is given by
\begin{equation}\label{defUpsilon}
\Upsilon(\slambda)=C\left(\slambda I-A_0-\sum_{i=1}^m A_i e^{-\slambda\tau_i}\right)^{-1} B.
\end{equation}
The $\mathcal{H}_2$ norm of $\Upsilon$ is defined in the frequency domain as
\[
\|\Upsilon\|_2
=\frac{1}{2\pi} 
\left(\int_{-\infty}^{\infty} \mathrm{Tr}\left( \Upsilon^*(\imath\omega) \Upsilon(\imath\omega)\right)d\omega\right)^{\frac{1}{2}},
\]
while an equivalent definition in the time-domain is given by
\[
\|\Upsilon\|_2= \left(\int_{0}^{\infty} \mathrm{Tr}\left( h^T(t) h(t)\right)dt\right)^{\frac{1}{2}},
\]
with $h$ the impulse response. The following proposition expresses the $\mathcal{H}_2$ norm in terms of the delay Lyapunov matrix, whose proof trivially follows from the identity 
\begin{equation}\label{relKh}
h(t)=CK(t) B.
\end{equation}
\begin{proposition} \cite[Theorem~1]{jarlebringtac} The $\mathcal{H}_2$ norm of (\ref{defUpsilon}) satisfies
\begin{equation}\label{prop-cor1}
\|\Upsilon\|_2^2=\mathrm{Tr}\left(C P(0) C^T\right),
\end{equation}
where $P(t)$ is the delay Lyapunov matrix associated with matrix $B B^T$.
\end{proposition}

The aim of this paper is to present a novel method for computing delay Lyapunov matrices and $\mathcal{H}_2$ norms with the following properties:
\begin{itemize}
	\item it is generally applicable, in the sense that there are no restrictions on the number and values of the delays, and the delay Lyapunov matrix can be easily computed or extended a posteriori beyond the interval $[-\tau_m,\ \tau_m]$; 
	\item the number of operations scales favorably with respect to the 
	dimension $n$ of the system matrices, particularly if the matrices are sparse, targeting (discretizations of)  partial differential equations (PDEs) with delay, provided that the rank $r$ of $B$ is small compared to $n$ 
\end{itemize}
In the description of the results we restrict ourselves to the computation of Lyapunov matrix $P(t)$, since $Q(t)$  can be obtained from  Lyapunov matrix  $P(t)$ associated with a ``transposed" system, inferred from the substitutions $A_k\leftarrow A_k^T,\ k=0,\ldots,m $ and $B\leftarrow C^T$. The latter directly follows from a comparison between (\ref{BVP}) and (\ref{BVPdual}).

The characterization (\ref{BVP}) provides a natural way to compute the Lyapunov matrix, and the $\mathcal{H}_2$ norm via formula (\ref{prop-cor1}). However, there are major challenges. First, when making the leap from ordinary to delay differential equations, the algebraic Lyapunov matrix equation is replaced by a \emph{matrix-valued} boundary value problem with delay. Second, bringing the equation in triangular form using a Schur decomposition, which forms the basis of the celebrated Bartels-Stewart algorithm for the matrix Lyapunov equation, is no longer possible.  
Third, it has been shown in \cite{jarlebringtac} that function $\mathbb{R} \ni t\mapsto P(t)$ may be non-smooth. The function is continuous, but it may be not be differentiable at $t=0$. On the interval $[0,\ \infty)$, to which we restrict in this paper because of the second condition in (\ref{BVP}), it is continuously differentiable, yet the second derivative might be discontinuous at $t=\tau_i,\ i=1,\ldots,m$, as we shall illustrate in the next section.

In the present literature two approaches for solving (\ref{BVP}) can be identified.  
The first one, the so-called direct approach, is based on approximating the solution on an interval by a matrix polynomial or a piecewise matrix polynomial and, besides imposing the boundary conditions and continuity requirements, determining the coefficients by collocation conditions for the differential equation, see, e.g., \cite{Huescaetal2009,jarlebringtac}.  With $N$ the number of collocation points, this results in a linear system of equations in $\mathcal{O}(n^2 N)$ variables. The convergence of the obtained approximations to the solution as a function of $N$ might be slowed down by the lack of smoothness of the solution discussed above, see \cite{jarlebringtac} for a detailed analysis.
The second  approach can be interpreted as a shooting method. It is applicable only if the time-delays are commensurate,\ i.e.,\ $\tau_i= n_i h$ for some $h>0$ and $n_i\in\NN,\  i=1,\ldots,m$, and it exploits that, in this case, the solution of (\ref{BVP}) is piecewise smooth (more precisely, smooth on intervals of form $(i h,\ (i+1)h),\ i\in\mathbb{Z}$).  Then (\ref{BVP}) can be reformulated as a standard boundary value problem for an ordinary differential equation of dimensions $2n^2 n_m$ on the interval $[0,\ h]$. For the latter boundary value problem, the transition between starting and end time can be determined explicitly in the form of the action of a matrix exponential (the so-called semi-analytic approach \cite{kharmatrix,wimmarco}) or by a numerical time-integration scheme \cite{jarpol}.

The common factor that leads to a poor scalability of the above vectorization based approaches with respect to the dimension $n$ is that they rely on solving a system of equations in $n^2$ variables, possibly multiplied with a large factor, hence when using  a direct solver the number of elementary operations amounts to $\mathcal{O}(n^6)$ operations.   To the best of the authors' knowledge the only available method that allows to address large-scale problems is the one presented in \cite{jarpol} for the single delay case, which falls under the umbrella of shooting methods, with the transition map determined by time-integration.  The key idea behind this approach, which has been shown to be effective for problems with  $n$ up to $\approx 1000$, is to solve the linear system of equations arising from the shooting method using a preconditioned Krylov method, where the preconditioner  is determined from the corresponding problem without delay. The latter allows an application of the preconditioner using $\mathcal{O}(n^3)$ operations.  This approach is complementary to the presented approach, which has the distinctive feature that it grounded in solving standard Lypunov matrix equations.


In Section~\ref{secfinal} we present a spectral discretization of equation (\ref{sys}) into an ordinary equation of dimensions $(N+1)n$, where $N$ determines the resolution of  the discretization. This allows us to obtain approximations of delay Lyapunov matrices and $\mathcal{H}_2$ norms from solving standard Lyapunov matrix equations. We also show how using a transformation a favorable structure can be imposed.
The main results are obtained in Section~\ref{secproject}, where among others projections on Krylov spaces are used to approximate the solutions of these Lyapunov equations, resulting in a dynamic construction of Lyapunov matrix approximations. Note that Krylov methods constitute an established approach for solving large scale matrix Lyapunov equations, see, e.g., \cite{simoncinisirev,simoncini,druskin} and the references therein.
We will show that several methodological choices can be made in such a way that the overall algorithm does not depend any more on parameter $N$ (the only condition is that it is sufficiently large with respect to the number of iterations for building the Krylov space). This property is at the basis of an interpretation of in terms of a projection method  applied directly to a linear infinite-dimensional system equivalent  to the original delay system.  In this sense the algorithm complements the set of ``discretizaton free" algorithms for solving nonlinear eigenvalue problems and associated problems in \cite{jarlebring-sisc,jarlebring-tensor}.
Numerical experiments are reported in Section~\ref{secapplic} and some concluding remark are formulated in Section~\ref{secconcl}. Preliminary results regarding the computation of $\mathcal{H}_2$ norms have been presented in~\cite{jelle}.

\section{Finite-dimensional approximation}\label{secfinal}
In Section 2.1 we outline how to discretize (\ref{sys}) (and as a consequence (\ref{sys-z})) using a spectral method \cite{trefethenspectral}, resulting in a system described by ordinary differential equations. For sake of conciseness, the derivation is slightly different from \cite{breda}, in the sense that the connection of (\ref{sys}) with an abstract infinite-dimensional linear system is not explicitly made. Subsequently, we outline how an approximation of the delay Lyapunov matrix can be obtained from this discretization.  In Section 2.2 we discuss and illustrate properties of the obtained approximations.  In Section~\ref{sparse} we refermulate the 
expressions for the delay Lyapunov approximations in a form that is more suitable for the application of a Krylov method.

\subsection{A spectral discretization}\label{parspect}

 Given a positive integer
$N$, we consider a mesh $\Omega_N$ of $N+1$ distinct
points in the interval $[-\tau_{m},\ 0]$:
\begin{equation}\label{defmesh2}
\Omega_N=\left\{\theta_{N,i},\ i=1,\ldots,N+1\right\},
\end{equation}
where
\[
-\tau_{m}\leq\theta_{N,1}<\ldots<
\theta_{N,N}<\theta_{N,N+1}=0.
\]
Throughout the paper we choose the nonzero mesh points as scaled and shifted zeros of  the Chebyshev polynomial of the second kind and order $N$, i.e. the mesh points are specified as
\begin{equation}\label{defgrid}
\theta_{N,i}=\frac{\tau_m}{2}(\alpha_{N,i}-1),\ \   \alpha_{N,i}=-\cos\frac{\pi i}{N+1},\ i=1,\ldots,N+1.
\end{equation}

Denoting with $l_{N,k}$ the Lagrange polynomials corresponding to $\Omega_N$, i.e., real valued
polynomials of degree $N$ satisfying
\[
l_{N,k}(\theta_{N,i})=\left\{\begin{array}{ll}1 & i=k,\\
0 & i\neq k,
\end{array}\right.
\]
and letting $x_k,\ k=1,\ldots,N+1$ functions from $\RR$ to $\RR^n$, we approximate the ``piece of trajectory" $x(t+\theta),\ \theta\in [-\tau_m,\ 0]$ as follows,
\begin{equation}\label{rightapprox}
x(t+\theta)\approx \sum_{k=1}^{N+1} l_{N,k}(\theta) x_k(t),\ \ \theta\in[-\tau_m,\ 0],
\end{equation}
which induces on its term the approximation
\begin{equation}\label{connections}
\left\{\begin{array}{rcl}
x_1(t) &\approx & x(t+\theta_{N,1}),\\
&\vdots & \\
x_{N}(t)&\approx &x(t+\theta_{N,N}),\\
x_{N+1}(t)& \approx  &x(t).
\end{array}\right.
\end{equation}
Along a solution of (\ref{sys-z}), $x$ is differentiable almost everywhere, hence for almost all $t\geq 0,\ \theta\in[-\tau_m,\ 0]$ we can express
\[
\frac{\partial x(t+\theta)}{\partial t}=\frac{\partial x(t+\theta)}{\partial \theta}.
\] 
Requiring that the right-hand side of (\ref{rightapprox}) satisfies this identity for (collocation points) $\theta_{N,1},\ldots,\theta_{N,N}$ brings us to the equations
\begin{equation}\label{constr1}
\dot x_i(t)= \sum_{k=1}^{N+1} \dot l_{N,k}(\theta_i) x_k(t),\ \ \  i=1,\ldots,N.
\end{equation}
Next, substituting the right-hand side of (\ref{rightapprox}) into (\ref{sys}) yields
\begin{equation}\label{constr2}
\left\{\begin{array}{lll}
\dot x_{N+1}(t) &=& A_0 x_{N+1}(t)+ \left(\sum_{i=1}^m  \sum_{k=1}^{N+1}  A_i l_{N,k}(-\tau_i) \right) x_k(t)+ B u(t),\\
y(t) &=& C x_{N+1}(t).
\end{array}\right.
\end{equation}

Letting $z(t)=[x_1^T(t)\ \cdots\ x_{N+1}^T(t)]^T\in\RR^{(N+1)n\times 1}$,  Equations (\ref{constr1}) and (\ref{constr2}) can be written  as 
\begin{equation}\label{approx}
\left\{\begin{array}{l}
\dot z(t)= \mathcal{A}_N z(t)+B_N u(t), \\
y(t)= C_N z(t),
\end{array}\right.
\end{equation}
where
\begin{equation}\label{defAN}
\begin{array}{l}
\mathcal{A}_N=
\left[\begin{array}{lll}
d_{1,1} &\hdots & d_{1,N+1} \\
\vdots & & \vdots \\
d_{N,1} &\hdots & d_{N,N+1} \\
a_{1} & \hdots & a_{N+1}\\
\end{array}\right],
\ \ 
B_N=\left[\begin{array}{c}0\\ \vdots\\ 0 \\1  \end{array}\right]\otimes B,
\\
 C_N=[0\ \cdots\ 0\ 1]\otimes C
\end{array}
\end{equation}
and
\[
\left\{\begin{array}{llll}
d_{i,k}&=&\dot l_{N,k}(\theta_{N,i}) I_n,\ \ \ \
& i\in\{1,\ldots,N\},\ k\in\{1,\ldots,N+1\}, \\
a_{k}&=&A_0 l_{N,k}(0)+\sum_{i=1}^m A_i
l_{N,k}(-\tau_i),\ \ \ &  k\in\{1,\ldots,N+1\}.
\end{array}\right.
\]
The advantage of approximation (\ref{approx}) is that it is in the form of a standard state space representation, for which many analysis and control design techniques exist. We refer to \cite{jorisijc} where (\ref{approx}) is at the basis of a design method for fixed-order $\mathcal{H}_2$ optimal controller.

According to (\ref{connections}), it is natural to relate the initial condition in the definition of the fundamental solution $K$, see (\ref{defK}), with initial condition  $z(0)=E_N$ of (\ref{approx}), where
\[
E_N=[0\ \cdots\ 0\ 1]^T \kron I_n.
\]
This  allows us to approximate the fundamental matrix $K(t)$ by $K_N(t)$, defined as
\begin{equation}\label{defKN}
K_N(t)=E_N^T e^{\mathcal{A}_N t} E_N,
\end{equation}
which by (\ref{defU}) leads us on its turn to an approximation $\mathcal{P}_N$ of $P$, 
\begin{equation}\label{defUN}
\begin{array}{lll}
	\mathcal{P}_N(t)
	&= &
	\int_{0}^{\infty} K_N(s) B B^T K(s+t) ds
	\\
	&=&\int_{0}^{\infty} E_N^T e^{\mathcal{A}_N s} B_N B_N^T e^{\mathcal{A}_N^T(s+t)}E_N ds.
	\end{array}
\end{equation}
Similarly, we can approximate $\Upsilon$ in (\ref{defUpsilon}) by the transfer function of (\ref{approx}), given by
\begin{equation}\label{defUpsilonN}
\Upsilon_N(\slambda)=C_N\left(\slambda I-\mathcal{A}_N\right)^{-1}B_N.
\end{equation}
The following proposition provides a computational expression for $\mathcal{P}_N$ in terms of a Lyapunov matrix equation. The arguments in the proof are well known but we include them to make the paper self contained.
\begin{proposition} If matrix $\mathcal{A}_N$ is Hurwitz, we can express
\begin{equation}\label{exprUN}
\mathcal{P}_N(t)=E_N^T P_N e^{\mathcal{A}_N^T t} E_N,	
\end{equation}
where $P_N$ satisfies the Lyapunov equation
\begin{equation}\label{lyap-intro}
\mathcal{A}_N P_N+P_N \mathcal{A}_N^T+B_N B_N^T=0.
\end{equation}
\end{proposition}
\begin{proof}
We can write (\ref{defUN}) as $\mathcal{P}_N(t)= E_N^T \tilde P_N e^{\mathcal{A}_N^T t} E_N$, where
\[
\tilde P_N=\int_{0}^{\infty}  e^{\mathcal{A}_N s} B_N B_N^T e^{\mathcal{A}_N^T s} ds. 
\] 	
We have
\[
\begin{array}{lll}
\mathcal{A}_N \tilde P_N+ \tilde P_N \mathcal{A}_N^T &=& \int_{0}^{\infty} \frac{d}{ds}\left(e^{\mathcal{A}_N s} B_N B_N^T e^{\mathcal{A}_N^T s}   \right)ds
\\
&=& -B_N B_N^T,
\end{array}
\]
the latter following from the Hurwitz property of $\mathcal{A}_N$. Since for the same reason the solution to Lyapunov equation (\ref{lyap-intro}) uniquely exists, we conclude $\tilde P_N=P_N$.
\end{proof}

\smallskip

The next proposition expresses that the approximation (\ref{defUN}) of the delay Lyapunov matrix, and the approximation of the transfer function, are consistent with respect to property (\ref{prop-cor2}).
\begin{proposition} \label{prop-cor2} Function (\ref{defUN}) and transfer function (\ref{defUpsilonN}) satisfy
\begin{equation}\label{approxH2}
\|\Upsilon_N\|_2^2= \mathrm{tr} \left(C \mathcal{P}_N(0) C^T\right).  
\end{equation}
\end{proposition}

\subsection{Properties}\label{parprop}

We discuss properties of approximations (\ref{approx}) and (\ref{defUN}) which are instrumental to the developments in the next sections, and which further shed a light on the difficulty of the problem of computing the delay Lyapunov matrix. They are illustrated by means of the didactic example
\begin{equation}\label{didactic}
\dot x(t)=\frac{1}{2} x(t)-x(t-1)+u(t),\ y(t)=x(t).
\end{equation}

\paragraph{Time domain}
The function $t\mapsto K(t)$ is in general not analytic on $(0,\ \infty)$, due to the propagation of the discontinuity at $t=0$. If function $K$ has a discontinuity in its $k$-th derivative ($k=0$ for a discontinuity in the function) at some time $\hat t\geq 0$, then  the function has, in the generic case, a discontinuity in its $(k+1)$-th derivative at time instants $\hat t+\tau_i,\ i=1,\ldots,m$.  The increase of regularity is called the smoothing property of solutions \cite{hale}.

Via definition (\ref{defU}) the non-smoothness of $K$ propagates to the function $t\geq 0\mapsto P(t)$ (we restrict to non-negative $t$ because of the so-called symmetry property $P(-t)=P(t)^T$). In Section 4 of~\cite{jarlebringtac}~it has been shown that the function $P$ is in general not infinitely many times differentiable for $t\in S,$ where
\[
S=\left\{\vec\tau\cdot\vec z:\ \vec z\in\mathbb{Z}^m,\ \vec\tau\cdot\vec z>0  \right\},
\] 
where $\vec\tau=(\tau_1,\ldots,\tau_m)$ and $\vec z=(z_1,\ldots,z_m)$.  In the commensurate delay case, where $\vec\tau=h \vec n$ with $n\in\NN^m$ and $\gcd(\vec n)=1$, we have $S=\left\{k h:\ k=0,1,2,\ldots\right\}$. In case of non-commensurate delays, set $S$ is dense in $[0,\infty)$. In both cases,  function $P$ is continuous, $\dot P $ are continuous on $(0, \infty)$, while  $\ddot P$ is continuous for all $t\in (0,\infty)$ except for $t=\tau_i,\ i\in\{1,\ldots,m\}$, but still of bounded variation. For more details we refer to \cite{jarlebringtac}.  As an illustration we plot the functions $K$ and $P$, corresponding to (\ref{didactic}), in Figure~\ref{figK}.

\begin{figure}[h]
\resizebox{!}{4.7cm}{\includegraphics{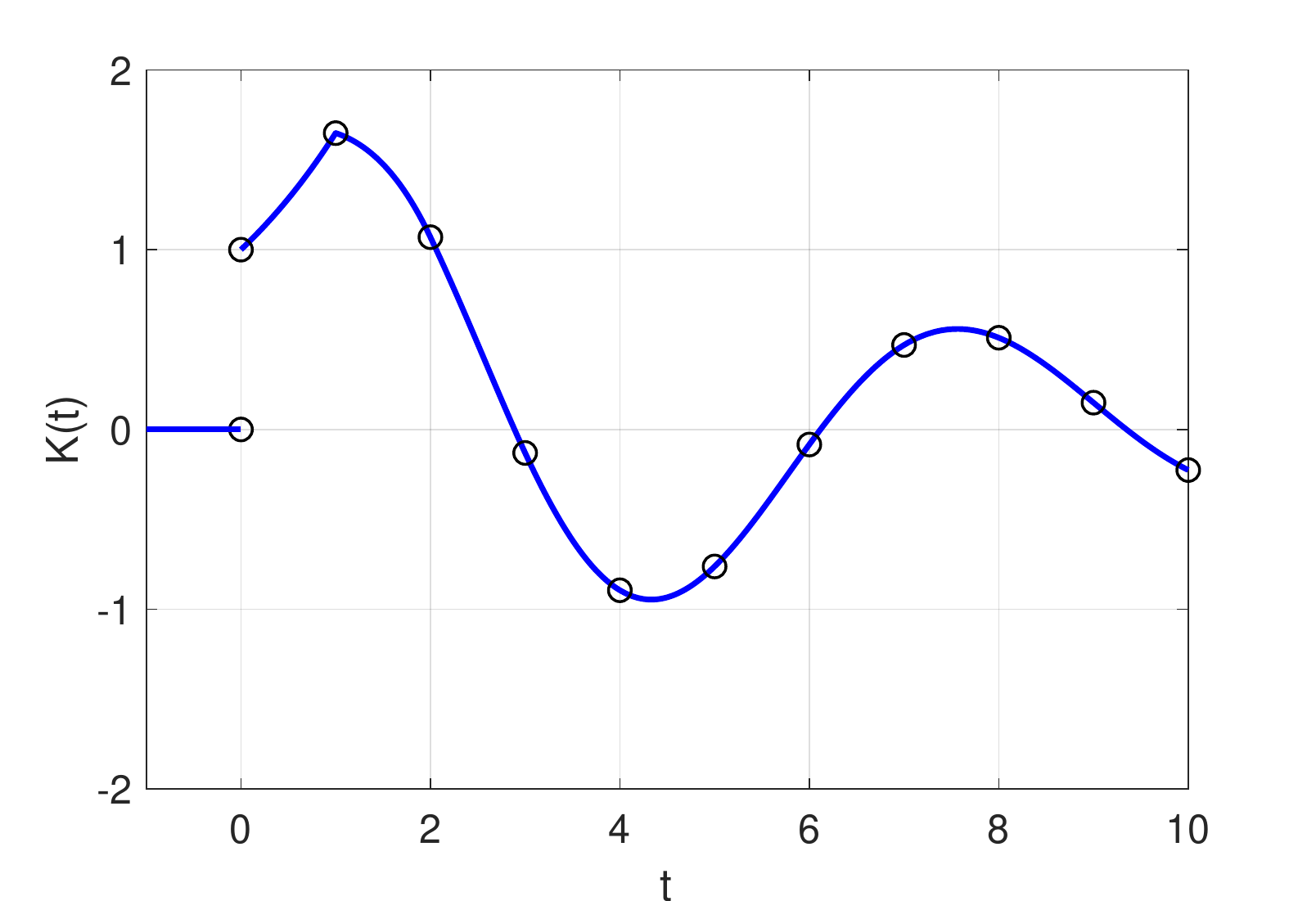}}
\resizebox{!}{4.7cm}{\includegraphics{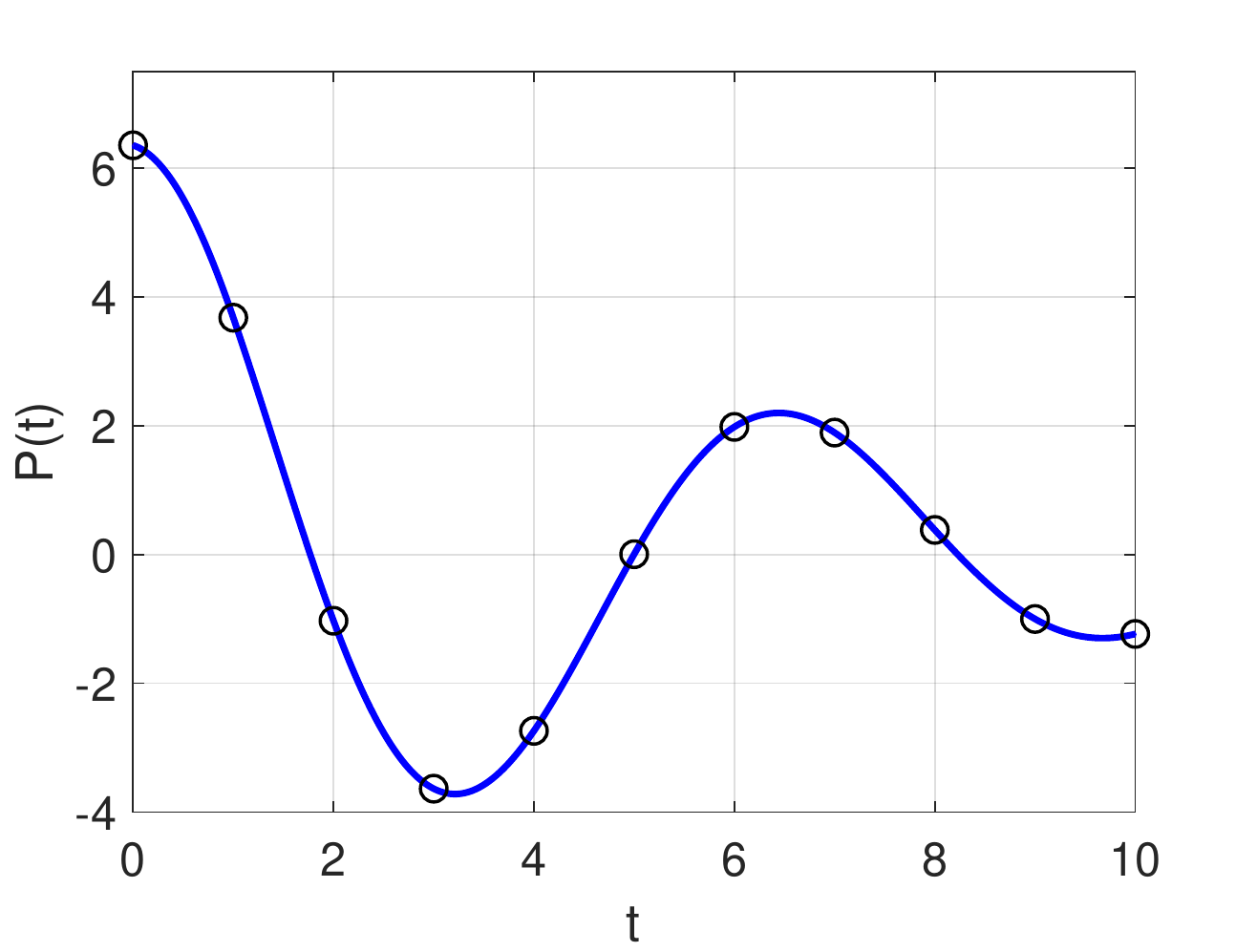}}
\caption{\label{figK} Plot of functions $K$ and $P$ for system (\ref{didactic}). The circles correspond to time-instants where the function is not infinitely many times differentiable. Function $K$ (function $P$) exhibits a discontinuity in its $k$-th derivative ($(k+1)$-th derivative) at $t=k$, for all $k\in\mathbb{N}$. }
\end{figure}

In Figure~\ref{figconv2} we plot for system (\ref{didactic}) the normalized errors
\begin{equation}\label{figer1}
\frac{\max_{t\in[0,\ t_{\max}]} |P(t)-\mathcal{P}_N(t)|}{\max_{t\in[0,\ t_{\max}]} |P(t)|} 
\end{equation}
for $t_{\max}=2$ and
\begin{equation}\label{figer2}
\frac{|P(0)-\mathcal{P}_N(0)|}{|P(0)|}, 
\end{equation}
as a function of $N$.  Note that, as $B=C=1$ for system (\ref{didactic}), expression (\ref{figer2}) corresponds to the normalized error on the squared $\mathcal{H}_2$ norm if the latter is approximated by $\|\Upsilon_N\|_2^2$, see (\ref{approxH2}). We observe the following rates of convergence: $\mathcal{O}\left(N^{-2}\right)$ for (\ref{figer1}), and $\mathcal{O}\left(N^{-3}\right)$ for (\ref{figer2}).  In all other experiments we observed the same rates of convergence.

The seemingly slow convergence, $\mathcal{O}\left(N^{-2}\right)$ for the maximum error of $P$ on a compact interval, is expected in view of the smoothness properties of function $P$. As we have seen, $P$ has discontinuities in its second derivative at $t=\tau_i,\ i=1,\ldots,m$ (with $\ddot P$  of bounded variation), while function $\mathcal{P}_N$, defined by (\ref{defUN}), is analytic on $\mathbb{R}$. Thus, we are approximating a non-smooth function by a series of smooth functions.  Note that we would  obtain the same rate of convergence when approximating $P$ on an interval by a series of polynomials interpolating in a Chebyshev mesh \cite[Theorem~7.2]{atap}.  As $P$ is analytic in the interval $(0,\tau_1)$ and we only consider nonnegative $t$, the convergence rate is better at $t=0$. We refer to \cite{jorisijc}, where an extensive argumentation for the rate $\mathcal{O}(N^{-3})$ for the $\mathcal{H}_2$ norm approximation induced by $\|\Upsilon_N\|_2$ is given.
We recall that the lack of smoothness of $P$ also affects solution schemes based on solving the boundary value problem (\ref{BVP}) directly \cite{jarlebringtac}.

\begin{figure}
	\begin{center}
	\resizebox{!}{6cm}{\includegraphics{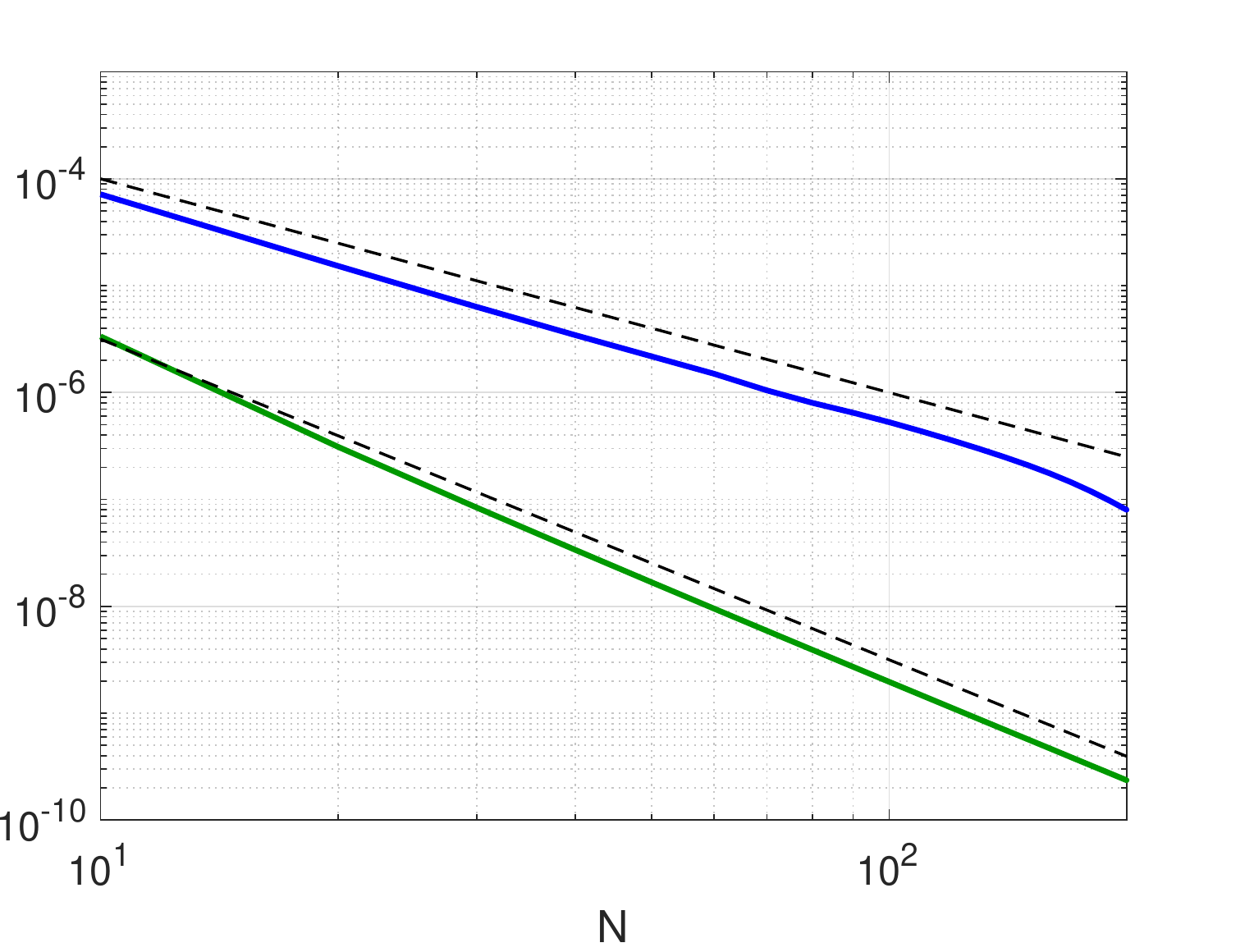}}
	\end{center}
	\caption{\label{figconv2} Normalized error (\ref{figer1}) (blue curve) and (\ref{figer2}) (green curve) as a function of $N$ for system (\ref{didactic}). The dashed lines indicate the rates $\mathcal{O}\left(N^{-2}\right)$ and $\mathcal{O}\left(N^{-3}\right)$. }
\end{figure}

\paragraph{Frequency domain}
With the choice of the Chebyshev mesh (\ref{defgrid}) the asymptotic convergence of the individual eigenvalues of $\mathcal{A}_N$ to corresponding characteristic roots
is fast. More specifically, in~\cite{breda} it is proven that spectral accuracy (approximation error $O(N^{-N})$) is obtained. 
An additional property of using mesh (\ref{defgrid}) for discretizing (\ref{sys}), observed in extensive numerical experiments, is that the eigenvalues of $\mathcal{A}_N$, which have not yet converged to corresponding characteristic roots of (\ref{sys-z}), are very often located to the left of the eigenvalues that have already converged, which is important with respect to preservation of stability. These properties are illustrated for system (\ref{didactic}) in Figure~\ref{figeig}. Finally, since the effect of the spectral discretization can be interpreted in terms of a rational approximation of functions $\lambda\to\exp(-\lambda\tau_i),\ i=1,\ldots,m$ around zero, see \cite{paperwu}, convergence is almost invariably  reached first for the smallest characteristic roots in modulus if $N$ is gradually increased.  Due to the characteristic shape on the spectrum of delay equation (exhibiting  infinite root chains extending in the left half plane, along which the imaginary part grows exponentially as a function of the real part, see
\cite{vyhlidal} for a detailed description), the rightmost, stability determining roots, are typically among the smallest characteristic roots.

\begin{figure}
	\resizebox{!}{4.9cm}{\includegraphics{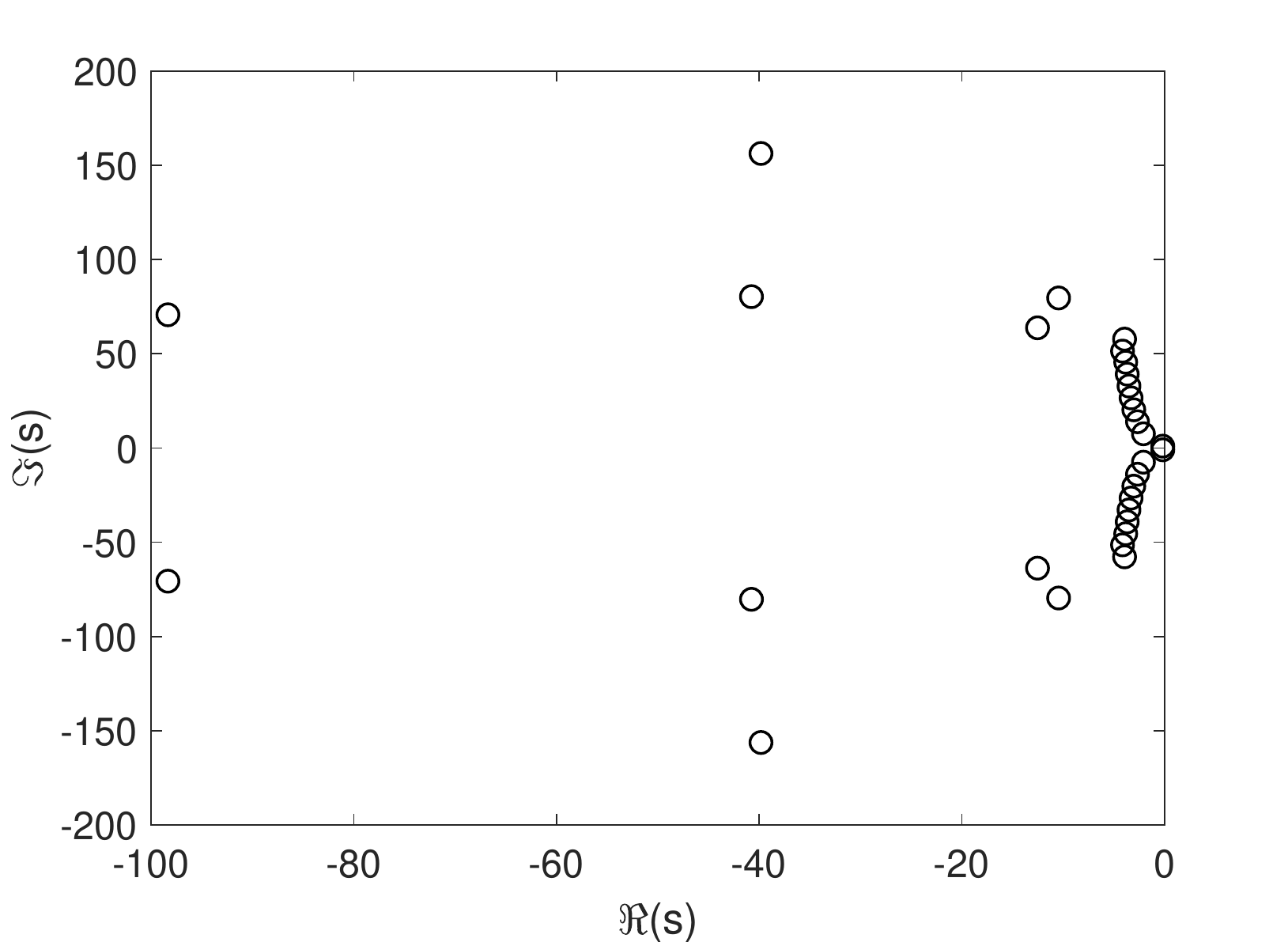}}
	\resizebox{!}{4.9cm}{\includegraphics{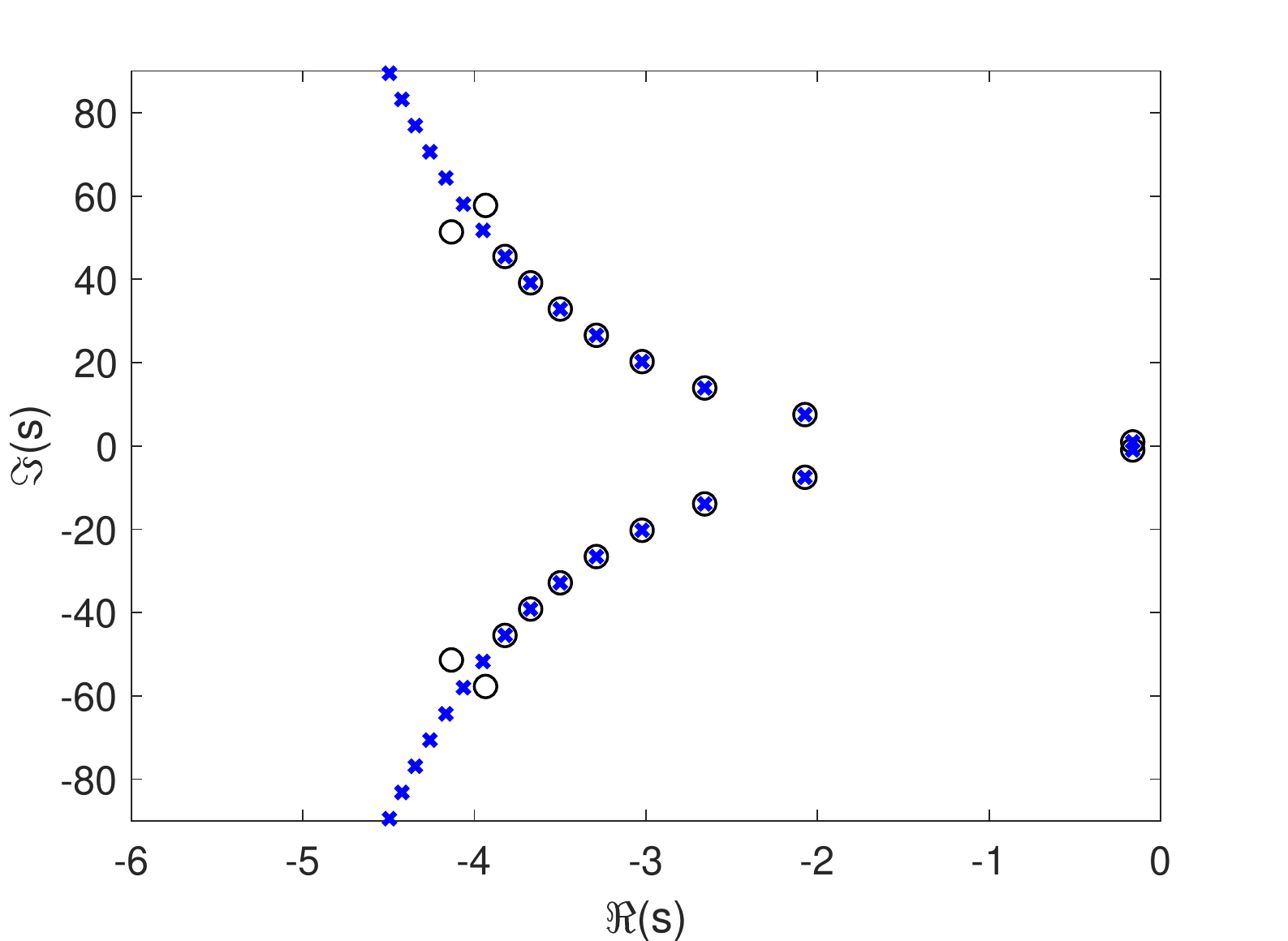}}
	\caption{\label{figeig}  (left) all eigenvalues of $\mathcal{A}_N$, corresponding to system (\ref{didactic}), for $N=30$ (black circles). (right)  Zoom of the right-part of the spectrum of $\mathcal{A}_N$, supplemented with the characteristic roots of the delay equation (blue stars). Its null solution is exponentially stable, with rightmost characteristic roots $-0.1629 \pm 0.9725j$. }
\end{figure}

With respect to the approximation of the transfer function,  the following moment matching property is proven in \cite{wimmoment}, which is in fact independent of the choice of the mesh points in (\ref{defmesh2}).
\begin{proposition}\label{theoremmoment}
	The transfer functions (\ref{defUpsilon}) and (\ref{defUpsilonN}) satisfy,
	\begin{equation}\label{momentprop1}
	\left.\frac{d^i \Upsilon_N(\slambda)}{d\slambda^i}\right|_{\slambda=0}
	=\left.\frac{d^i \Upsilon(\slambda)}{d\slambda^i}\right|_{\slambda=0},\ \ i=0,\ldots, N,
	\end{equation}
	and
	\begin{equation}\label{momenprop2}
	\left.\frac{d^i \Upsilon_N(\slambda^{-1})}{d\slambda^i}\right|_{\slambda=0}
	=\left.\frac{d^i \Upsilon(\slambda^{-1})}{d\slambda^i}\right|_{\slambda=0},\ \ i=0,1,
	\end{equation}
	that is, the moments of $\Upsilon(\slambda)$ and $\Upsilon_N(\slambda)$ at zero match up to the $N$th moment, and the moments at infinity match up to the first moment.
\end{proposition}

By Property (\ref{momentprop1}), which corresponds to Hermite interpolation at $s=0$, the region in the complex plane where the approximation is accurate extends from the origin as $N$ is increased, consistently with the convergence behavior of characteristic root approximations sketched in the right pane of Figure~\ref{figeig} .  At  the same time, the asymptotic delay rate of the transfer function  for $\omega\rightarrow\infty$, which is described by $CB/\omega$, is captured by property (\ref{momenprop2}). Note that higher-order moments of (\ref{defUpsilon}) at infinity  are not well defined, which is related to the property that $s=\infty$ is an \emph{essential} singularity of $\Upsilon$.  As a consequence, the overall approximation error is mainly due to a mismatch in the mid-frequency range. This is illustrated in Figure~\ref{figtransfer}, where we compare the transfer function of (\ref{didactic}) and its approximation of form (\ref{defUpsilonN}). 
\begin{figure}
	\resizebox{!}{4.8cm}{\includegraphics{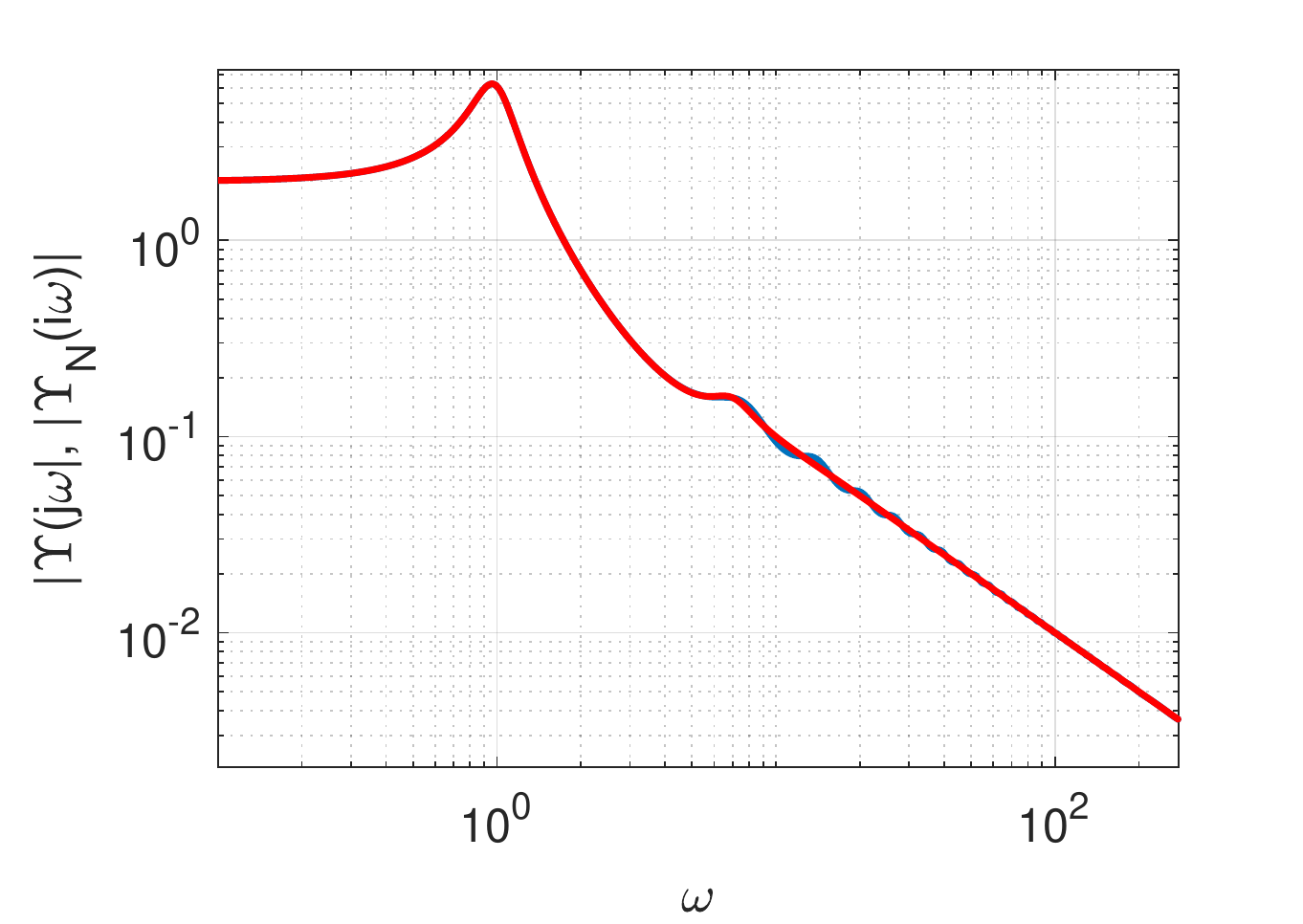}}
	\resizebox{!}{4.8cm}{\includegraphics{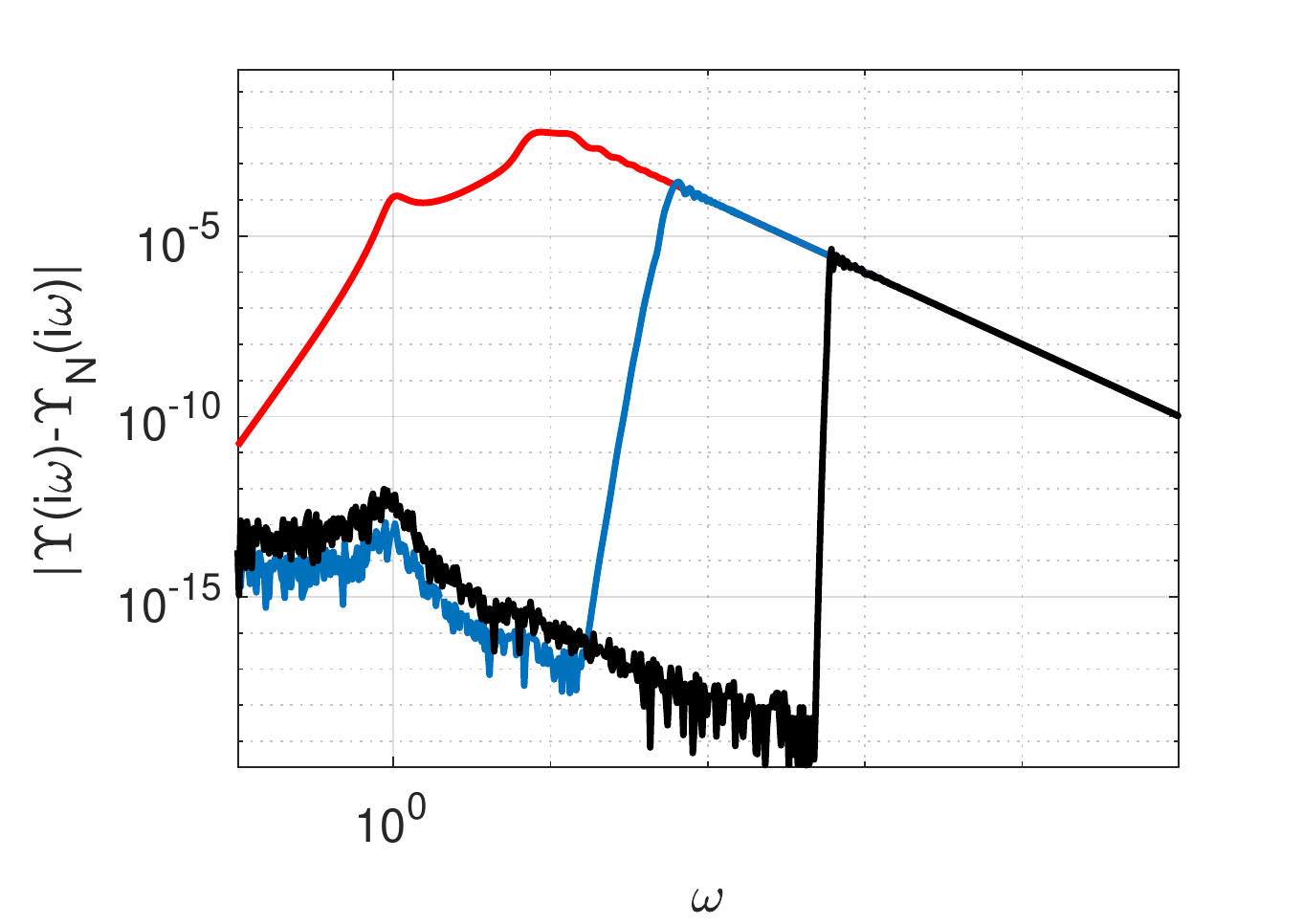}}
	\caption{\label{figtransfer} (left) Modulus of the transfer function of (\ref{didactic}) (blue curve) and the corresponding approximation (\ref{defUpsilonN}) for $N=5$ (red curve), evaluated on the imaginary axis, i.e.~for $s=\imath\omega,\ \omega\geq 0$. (right) Approximation error on the imaginary axis for $N=5$ (red curve), $N=30$ (blue curve) and $N=100$ (black curve).}
\end{figure}

The right pane in Figure~\ref{figtransfer} gives a complementary explanation, besides the smoothness properties of the function $t\mapsto P(t)$,  why the convergence of $\mathcal{P}_N(0)$ to $P(0)$ 
has exhibits a low rate of convergence $\mathcal{O}(N^{-3})$, compared to the spectral convergence of the eigenvalues of $\mathcal{A}_N$: unlike an individual pole and the $\mathcal{H}_{\infty }$ norm, the $\mathcal{H}_2$ norm is a \emph{global} characteristic of the transfer function, in the sense that an accurate computation involves approximating the transfer function well over whole the imaginary axis.

\subsection{A reformulation of the discretized problem}\label{sparse}

The following main theorem reformulates expressions (\ref{exprUN})-(\ref{lyap-intro}) in terms of a matrix $G_N$ similar to $\mathcal{A}_N^{-1}$, giving the Lyapunov equation a favorable structure that will be exploited by the algorithms presented in Section~\ref{secproject}.
\begin{theorem} \label{theosparse}
Assume that $\mathcal{A}_N$ is Hurwitz and let	
\begin{equation}\label{defGGN}
G_N=\Sigma_N^{-1}\Pi_N,
\end{equation}
where
{\small
	\begin{equation}\label{defPiN2}
	\Pi_N=\frac{\tau_{m}}{4}\left[\begin{array}{rrrrrrr}
	\frac{4}{\tau_m} &\frac{4}{\tau_m} & \frac{4}{\tau_m}&\cdots&&\cdots&\frac{4}{\tau_m}  \\
	2& 0&-1 &&&& \\
	&\frac{1}{2} &0 &-\frac{1}{2}&&&\\
	&&\frac{1}{3}&0&-\frac{1}{3}&&\\
	&&&\ddots &\ddots&\ddots&\\
	&&&        & \frac{1}{N-1}   &0& -\frac{1}{N-1}\\
	&&&            &      &\frac{1}{N} &0
	\end{array}\right]\otimes I
	\end{equation}
}
and
\begin{equation}\label{defSigmaN2}
\Sigma_N=\left[\begin{array}{cccc}
R_0 & R_1& \cdots&    R_N \\
& I_n &   & \\
&   & \ddots & \\
&   &        & I_n \\
\end{array}
\right],
\end{equation}
with
\[
R_i=A_0 T_i(1)+\sum_{k=1}^m A_k
T_i\left(-2\frac{\tau_k}{\tau_{m}}+1\right),\ i=0,\ldots,N
\]	
and $T_i$ the Chebyshev polynomial of the first kind and order $i,\ i=0,1,2,\ldots$. Moreover, let
\begin{equation}\label{defHN}
H_N=\left[\begin{array}{c}
R_0^{-1}\left(I-\frac{\tau_m}{2}R_1\right)R_0^{-1} B\\
\frac{\tau_m}{2}R_0^{-1} B \\
0\\
\vdots\\
0
\end{array}\right]
\end{equation}
and
\begin{equation}\label{defFN}
F_N=[ R_0\ R_1\ \ \cdots\  R_N].
\end{equation}
	
Then we can express $\mathcal{P}_N$ in (\ref{exprUN}) as
\begin{equation}\label{mainsparse}
\mathcal{P}_N(t)= F_N  Q_N e^{G_N^{-T}t} F_N^T
\end{equation}
where $Q_N$ satisfies the Lyapunov equation
\begin{equation}\label{lyapunov2}
 G_N Q_N+Q_N G_N^T+H_N H_N^T=0.
\end{equation}
Moreover, system (\ref{approx}) is equivalent to
\begin{equation}\label{approx-extra}
\left\{\begin{array}{l}
G_N \dot \eta(t)= \eta (t)+H_N u(t), \\
y(t)= C F_N \eta(t),
\end{array}\right.
\end{equation}
and we can express
\begin{equation}\label{transferlarge2}
\Upsilon_N(s)=C F_N(s G_N-I)^{-1} H_N.
\end{equation}
\end{theorem}

\begin{proof}
In \cite[Section 3.1]{wimmoment} it has been shown that
\begin{equation}\label{relAnPn2}
	\mathcal{A}_N= (S_N\otimes I) G_N^{-1} (S_N^{-1}\otimes I),
\end{equation}
where matrix $S_N\in\mathbb{R}^{(N+1)\times (N+1)}$ maps coefficients of a polynomial of degree $N$  in the Chebyshev basis
\begin{equation}\label{chebbasis2}
	\left\{ T_i\left(2 \frac{t}{\tau_m}+1\right):\ i=0,\ldots,N\right\}
\end{equation}
onto the corresponding coefficients in the Lagrange basis,
\[
\{l_{N,i}(t):\ i=1,\ldots,N+1\},
\]
defined on the mesh (\ref{defgrid}).	
	
Substituting (\ref{relAnPn2}) into (\ref{defUN}) yields
\begin{multline}\label{derUN1}
\mathcal{P}_N(t)=\int_{0}^{\infty} E_N^T (S_N\otimes I) e^{G_N^{-1} s}  (S_N^{-1}\otimes I) B_N  
\\
B_N^T  (S_N^{-T}\otimes I)e^{G_N^{-T}(s+t)}(S_N^T\otimes I) E_Nds.
\end{multline}	
In 	the proof of Theorem 3.2 of \cite{wimmoment} it has been shown that
\begin{equation}
(S_N^{-1}\otimes I)B_N=c_N\otimes B,\ \ E_N^T (S_N\otimes I)=\mathbf{1}_N^T\otimes I,
\end{equation}
with
\[
c_N=\left\{\begin{array}{ll}
	\frac{2}{N+1}\ [0\ 1\ 0\ 1\ \cdots\ 0\ 1]^T\otimes B, & N\  \mathrm{odd}, \\
	\frac{2}{N+1}\ [\frac{1}{2}\ 0\ 1\ 0\ 1\ \cdots\ 0\ 1]^T\otimes B, & N\  \mathrm{even}, \\
	\end{array}\right.
\]
and
$\mathbf{1}_N=[1\ 1\ \cdots\ 1]^T.	
$
Using these expressions, as well as the identity $e^{G_N^{-1}t}= G_N^{-1} e^{G_N^{-1}t} G_N$, we can write (\ref{derUN1}) as
\begin{multline}\label{derUN2}
\mathcal{P}_N(t)=\int_{0}^{\infty} (\mathbf{1}_N^T\otimes I)G_N^{-1} e^{G_N^{-1} s} G_N (c_N\otimes B) 
\\
(c_N^T\otimes B^T)G_N^T e^{G_N^{-T}(s+t)} G_N^{-T}(\mathbf{1}_N\otimes I) ds.	
\end{multline}	
A straightforward computation shows that
\[
(\mathbf{1}_N^T\otimes I) G_N^{-1}=F_N,\ \ \   G_N (c_N\otimes B)=\hat H_N,
\]	
with 
\[
\hat H_N=\left[\begin{array}{c} R_0^{-1} B\\0\\ \vdots \\ 0 \end{array}\right].
\]
As a consequence, we can write 	
\begin{equation}\label{exprUN4}
\mathcal{P}_N(t)= F_N \left(\int_{0}^{\infty} e^{G_N^{-1} s} \hat H_N  \hat H_N^T e^{G_N^{-T}s}~ds\right) e^{G_N^{-T}t} F_N^T.
\end{equation}
Denoting the integral in (\ref{exprUN4}) by $\hat Q_N$, we can express the latter, relying on the assumption that $\mathcal{A}_N$ 
and $G_N^{-1}$ are Hurwitz, as the solution of the Lyapunov equation
\[
G_{N}^{-1} \hat Q_N+ \hat Q_N G_N^{-T}+ \hat H_N \hat H_N^T=0.
\]
Pre-multiplying this equations with $G_N$ and post-multiplying with $G_N^T$ yields
\[
\hat Q_N G_N^T+ G_N \hat Q_N + G_N \hat H_N \hat H_N^T G_N^T=0.
\]
Since we have $G_N\hat H_N=H_N$, it follows that $\hat Q_N=Q_N$, where $Q_N$ uniquely solves (\ref{lyapunov2}). Hence, (\ref{exprUN4}) corresponds to (\ref{mainsparse}) and (\ref{lyapunov2}).

Finally,expression (\ref{transferlarge2}) constitutes the assertion of Theorem~3.2 of \cite{wimmoment}.
\end{proof}

\medskip

Matrices $\Sigma_N$ and $\Pi_N$ have a sparse structure that can be  exploited. In what follows a key role will be played by the following property.
\begin{proposition}\label{propdyn}
	Assume that $N_1,N_2\in\NN$ with $N_1<N_2$. Then the matrices $\Sigma_{N_1},\Pi_{N_1},F_{N_1},H_{N_1}$ in Theorem~\ref{theosparse} are submatrices of
	$\Sigma_{N_2},\Pi_{N_2},F_{N_2},H_{N_2}$.
\end{proposition}

\section{A dynamic subspace method} \label{secproject}

The price to pay for the discretization of the delay equation and the standard state space representation (\ref{approx}), which on their turn led us to  delay Lyapunov matrix approximations  in explicit form, namely (\ref{exprUN})-(\ref{lyap-intro})  and (\ref{mainsparse})-(\ref{lyapunov2}),  is an increase of dimension from $n$ to $(N+1)n$. At the same time relatively of high value of $N$ are expected for an accurate approximation, as motivated in Section~\ref{parprop}.  

 If $Nn$ is large and  matrix $B_N B_N^T$, respectively $H_N H_N^T$, has low rank (in the sense of $r<< Nn$), computing a low-rank approximation of $P_N$, respectively $Q_N$,  may be beneficial.  In  this section we construct an approximation inferred from the projection of the Lyapunov equation on a  Krylov space of dimension $kr$. Before we present the construction in Sections~\ref{pardyn}-\ref{pardynlyap}, we use another didactic example to motivate important methodological choices regarding 
 i) the relation between parameters $N$ and $k$, ii)  the choice of the Krylov space, and iii) the system matrix / Lyapunov equation to be projected on this space.
 We discuss some implementation aspects in Section \ref{parimpl} and conclude with an interpretation in terms of projecting an infinite-dimension system linear in Section~\ref{secinfcheb}. 
 
 The main contributions are contained in Sections~\ref{pardynlyap}-\ref{secinfcheb}. The Arnoldi process of Section~\ref{pardyn}  	 and the construction of the reduced model in Sections~\ref{pardyn2} extend results presented in \cite{jarlebring-sisc,wimmoment} to the multiple-input setting.

Since the technical derivations involve many steps, we included Figure~\ref{figoverview} at the end of the section in order to keep an overview of the main steps and corresponding notations.

\subsection{Motivation of methodological choices}
We consider system
\begin{equation}\label{didactic2}
{\small
\begin{array}{lll}
\dot{x}(t) &=&
\left[\begin{array}{rrr}
-0.08  & -0.03 &   0.2\\
0.2  & -0.04  & -0.005\\
-0.06  &  0.2  & -0.07
\end{array}\right]
x(t)+
\left[\begin{array}{rrr}
-0.0471  & -0.0504  & -0.0602\\
-0.0942  & -0.1008  & -0.1204\\
0.0471   & 0.0504  &  0.0602
\end{array}\right] x(t-5) 
\\
&&+\left[\begin{array}{r}1\\1 \\1 \end{array}\right]u(t),\ \ y(t)=\left[\begin{array}{rrr}1 & 0 & 0 \end{array}\right]x(t).
\end{array}
}
\end{equation}
For $N=50, 100, 150$ and $200$ we computed matrices $P_N$ and $Q_N$, solving Lyapunov equations~(\ref{lyap-intro}) and (\ref{lyapunov2}). We display in Figure~\ref{rankcond} (above) their ordered singular values, normalized such that the leading singular value equals to one.   We also show, in the lower figure, the leading singular value of both matrices as a function of $N$.
\begin{figure}
	\begin{center}
	\resizebox{6.4cm}{!}{\includegraphics{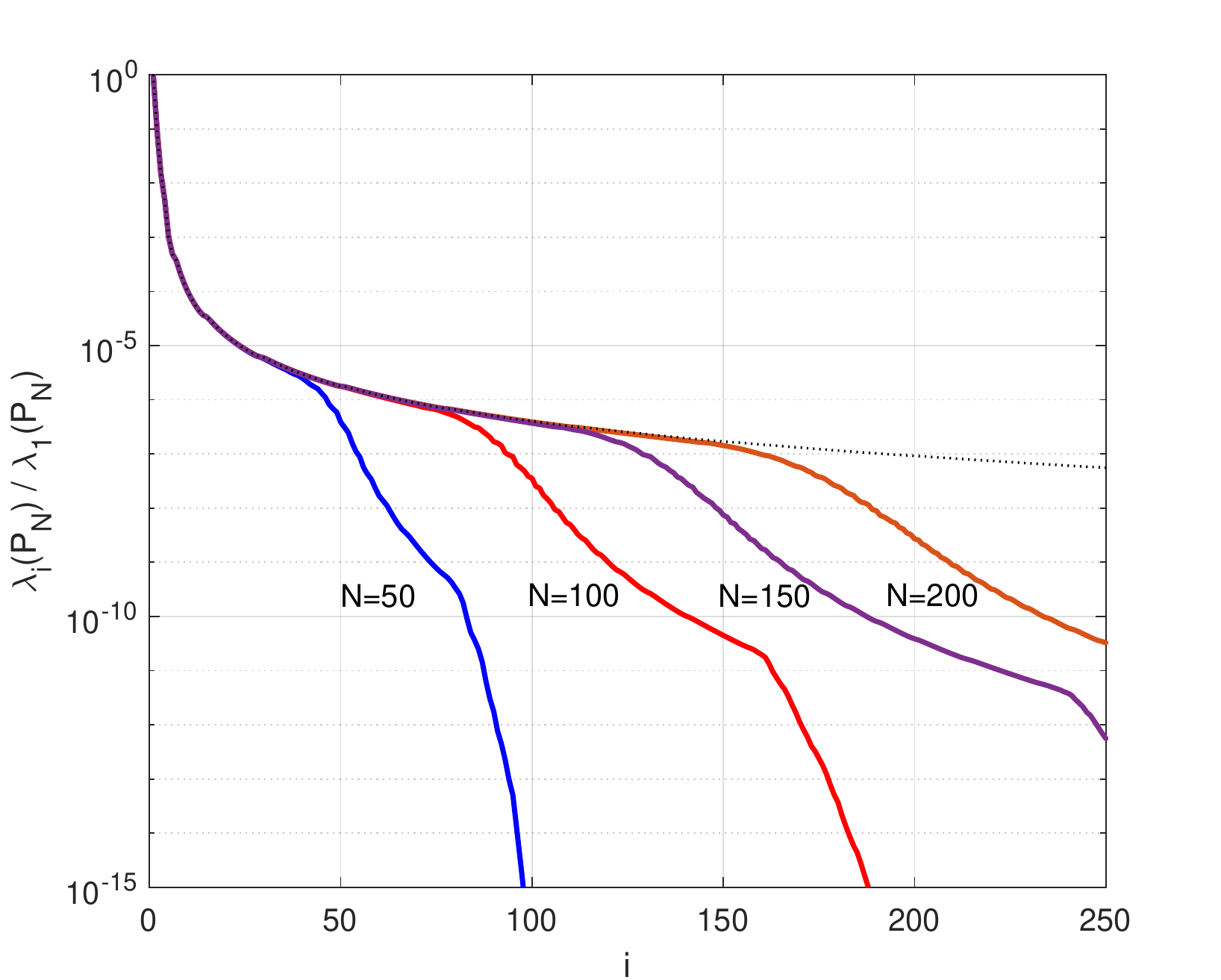}}
    \resizebox{6.4cm}{!}{\includegraphics{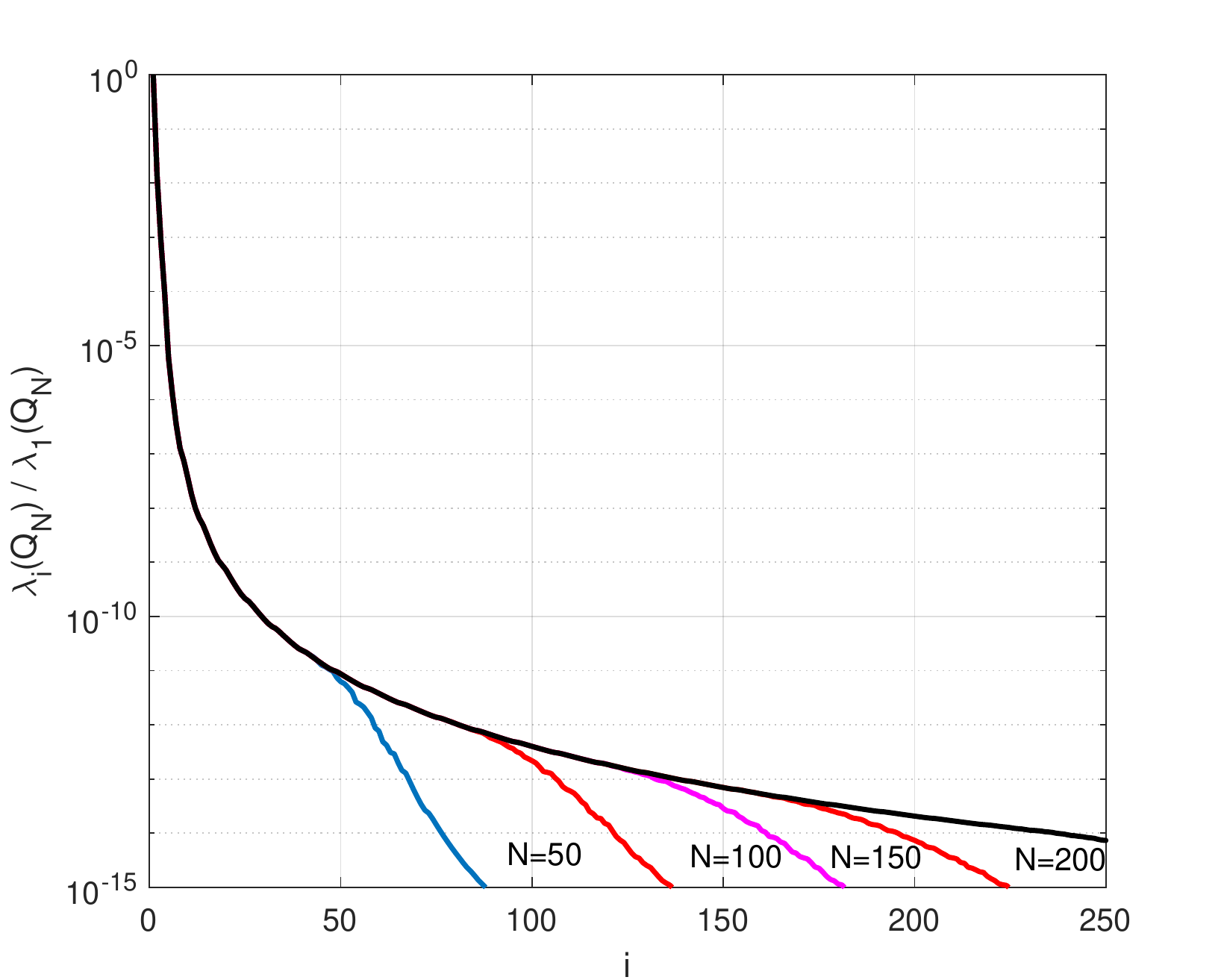}}
    \resizebox{6cm}{!}{\includegraphics{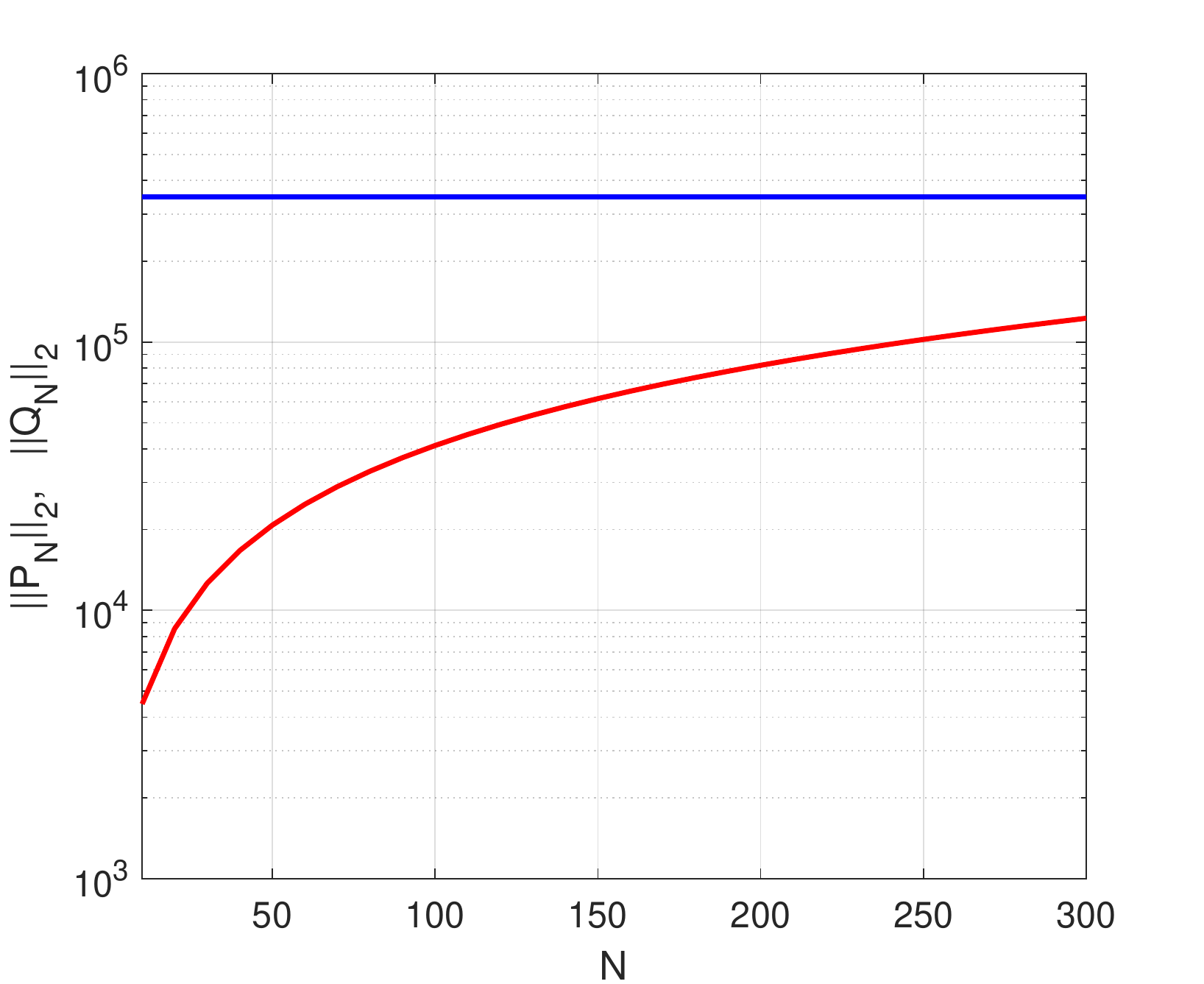}}
	\end{center}
\caption{\label{rankcond} (above) Normalized eigenvalues ($\lambda_i(\cdot)$ denoting the $i$-th, eigenvalue in decreasing order) of matrix $P_N$ and $Q_N$, computed for system~(\ref{didactic2}). (below) Spectral norm of $P_N$ (red) and $Q_N$ (blue) as a function of $N$. }
\end{figure}
%
%
This experiment indicates that the solution of Lyapunov equation (\ref{lyapunov2}), inferred from the representation (\ref{approx-extra}), is more amendable for a low-rank approximation.

Concerning the input-output behavior, Proposition~\ref{theoremmoment} expresses that functions $\Upsilon$ and $\Upsilon_N$ match $N+1$ moments at zero and two at infinity. To have these matching moments carried over by a projection of (\ref{approx}) on a right Krylov space, one needs in general a subspace of dimension $N+2$. At the same time, if more than $N+1$ moments at zero are preserved by the projection, or more than two at infinity, the highest order moments won't match anymore with those of the original transfer function (\ref{defUpsilon}). This can be interpreted as an instance of ``over-fitting" in the sense that particularities of the discretization (\ref{approx}) are captured by the projection, which are not present in the original delay equation and related to the discretization error.
Similar conclusions can be made from the experiment related to the upper right pane of Figure~\ref{rankcond}. On a compact interval for index $i$, the eigenvalue functions of $Q_N$ uniformly converge for $N\rightarrow\infty$ to the limit function indicated in black color, which is related to the original (non-discretized) delay equation (we come back to this in Section~\ref{secinfcheb}).   Important to observe is that, for a given value of $N$, \emph{less than $N$} singular values are related to the limit behavior. This indicates that, at least for a best rank-$k$ approximation of $Q_N$, the choice $k>N$ could lead to a similar instance of over-fitting. All the above elments  motivate us to assure $N$ being sufficiently large, such that 
the dimension of the subspace $k$ satisfies
  \begin{equation} \label{condN}
k\leq N
\end{equation}
and, preferably $k<<N$.

The typical spectrum distribution of delay equations, with rightmost characteristic roots close to the origin, the properties of the spectral discretization, illustrated in Figure~\ref{figeig}, and the above reasoning with respect to matching moments, suggest to build a Krylov space using matrix  $\mathcal{A}_N^{-1}$,
\[
\KK_{k}(A_{N}^{-1},B_N)=\operatorname{span}\left\{B_N,\mathcal{A}_N^{-1} B_N,\ldots,A_{N}^{-(k-1)}B_N\right\}.
\]
Letting the columns of $V_{N,k}$ be an orthogonal basis for this Krylov space, we depict in Figure~\ref{capture-eig} the approximation error on the smallest characteristic roots for (\ref{didactic2}), obtained as the reciprocal of the eigenvalues of 
\begin{equation}\label{projqual1}
V_{N,k}^T \mathcal{A}_N^{-1} V_{N,k}, 
\end{equation}
and as the eigenvalues of
\begin{equation}\label{projqual2}
V_{N,k}^T \mathcal{A}_N V_{N,k}, 
\end{equation}
for $N=30$, $N=60$  and in both cases $k=N$ (such that  (\ref{condN}) is taken into account). The plots illustrates a property observed in many experiments, that it is beneficial to project matrix $\mathcal{A}_N^{-1}$ on the Krylov space, compared to projecting $\mathcal{A}_N$. This observation can be explained by a better separation of the targeted characteristic roots after an inversion of the spectrum.
\begin{figure}
	\begin{center}
		\resizebox{!}{5cm}{\includegraphics{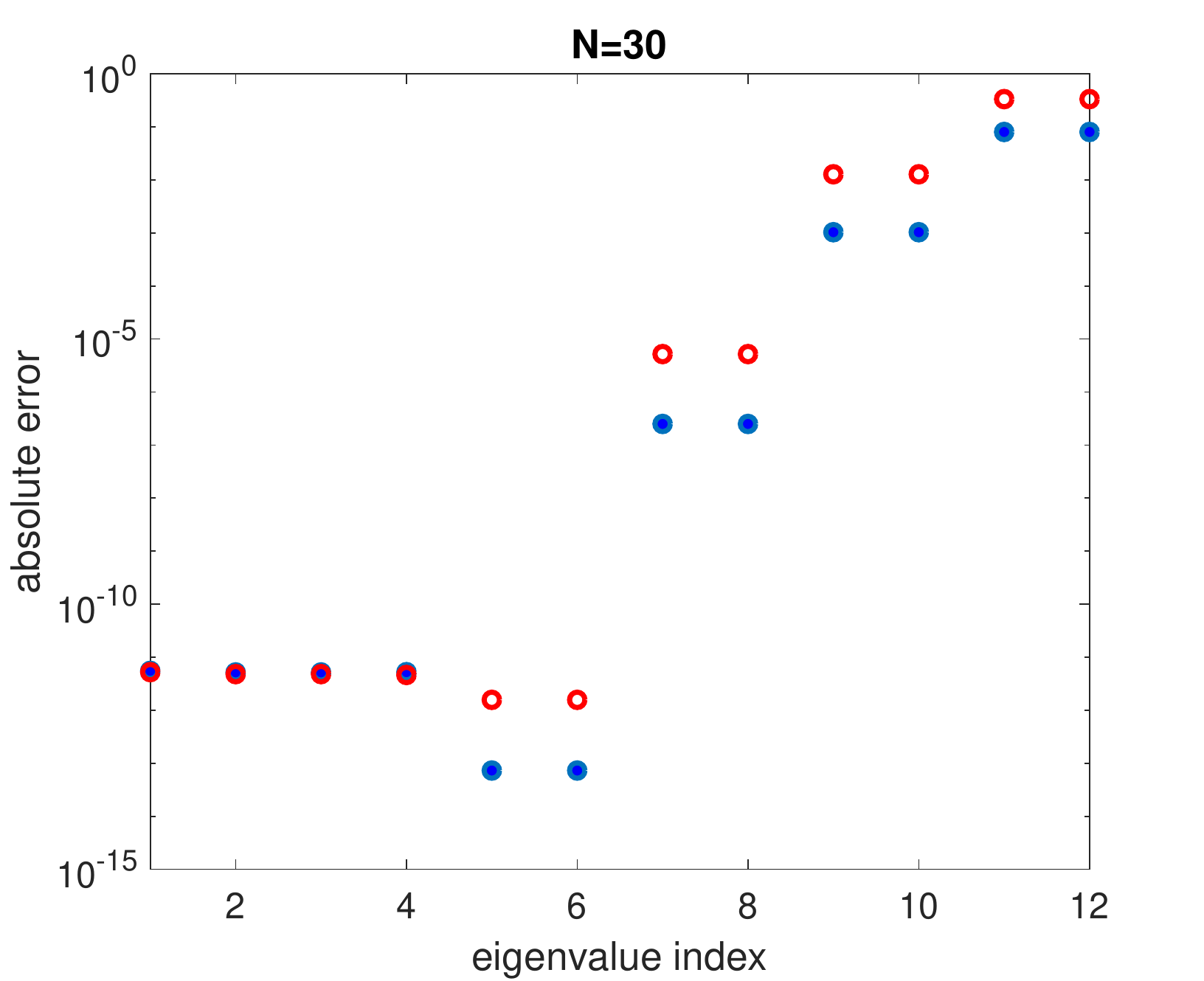}}
		\resizebox{!}{5.2cm}{\includegraphics{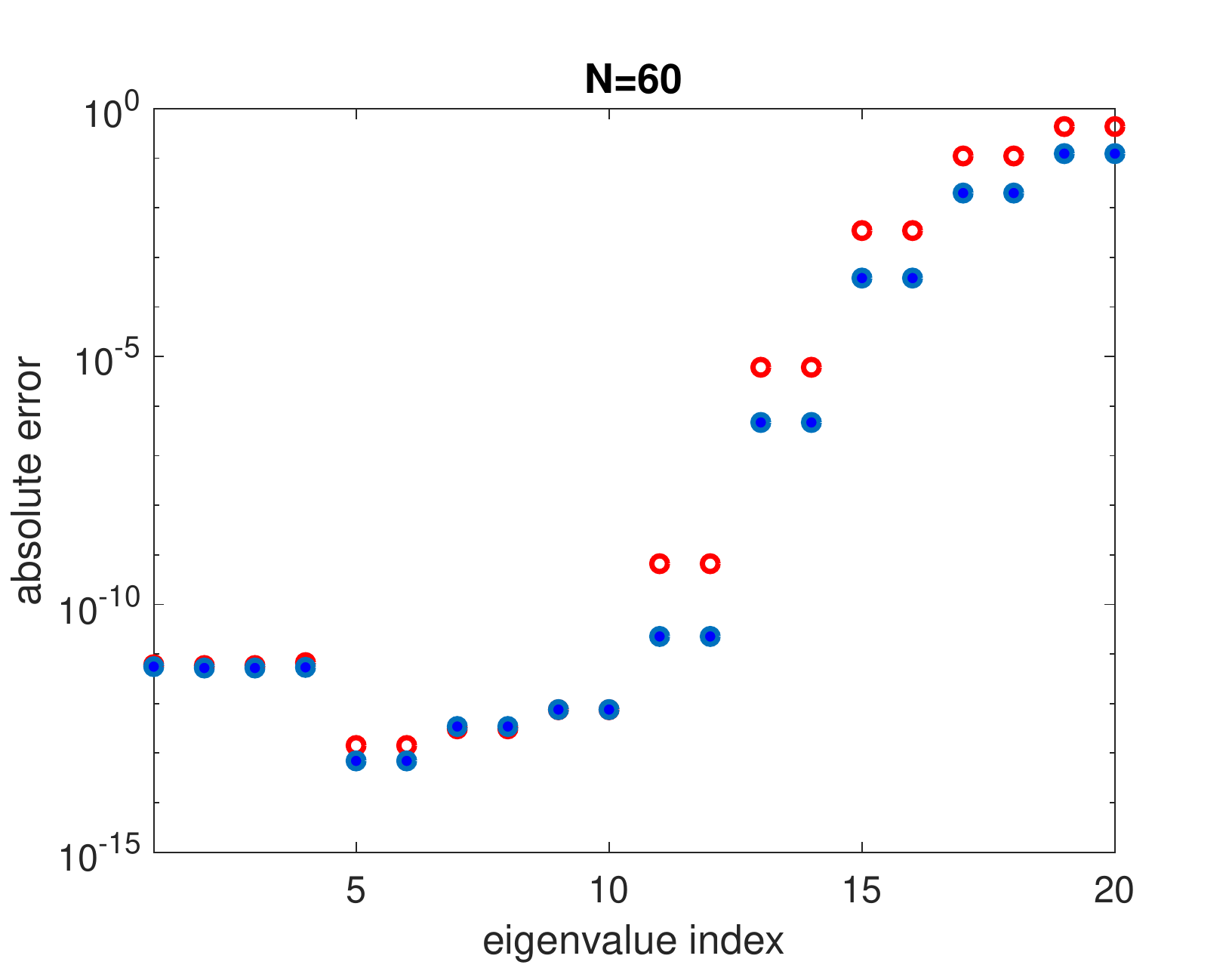}}
		\caption{\label{capture-eig} Absolute error on the smallest characteristic roots of (\ref{didactic2}), obtained from (\ref{projqual1}) (blue) and (\ref{projqual2}) (red). In the left pane we consider $N=k=30$, in the right pane $N=k=60$.}
	\end{center}
	\end{figure}

The preference for building a Krylov space for $\mathcal{A}_N^{-1}$ and  for projecting this matrix, the property that $G_N$ is similar to $\mathcal{A}_N^{-1}$, and, last but not least, the typically faster decay of singular values of $Q_N$ than those of $P_N$  naturally lead us to the representation 
(\ref{mainsparse})-(\ref{transferlarge2}) of the discretized system and associated approximation of the delay Lyapunov matrix. In addition, matrices $\Sigma_N$ and $\Pi_N$ have a sparse structure that can be exploited.   In particular, the property expressed in Proposition~\ref{propdyn} along with condition~(\ref{condN}) will allow us to ultimately arrive at a method  that does not rely on an a-priori choice of critical parameter $N$, similar to the infinite-Arnoldi method for eigenvalue computations~\cite{jarlebring-sisc}.

\subsection{Dynamic construction of a Krylov space}\label{pardyn}

We fix integer $k$ and assume~$N$ large enough such that (\ref{condN}) holds.
We  consider the block Krylov space
\begin{equation}\label{krylov}
\KK_{k}(G_{N},b)=\operatorname{span}\{b,G_{N} b,\ldots,G_{N}^{k-1}b\},
\end{equation}
where $b$ is a block vector of size $(N+1) n\times r$,  having the structure
\begin{equation}\label{startarnoldi}
b=[x_0^T\ 0\ \cdots 0]^T,\ \ 
\end{equation}
with $ x_0\in\RR^{n\times r}$ to be specified in Section~\ref{pardyn2}.
 The block Arnoldi algorithm builds the Krylov sequence, block vector by block vector, where these vectors are orthogonalized. 
Due to the special structure of $b$ and the fact that $G_N$ is a block Hessenberg matrix, whose blocks have size $n\times n$, the block vectors $G_{N} b,\ldots, G_{N}^{k-1}b$ only have their first $2n$, $3n,\ \ldots,\ k n$ block rows different from zero. 
Moreover, in computing the matrix vector products with (\ref{startarnoldi}), only sub-matrices of $G_N$ are needed. Hence, in the  computation of the Krylov space, we can restrict to storing only the nonzero part of the block vectors and using the relevant part of $G_N$.  This leads us to the following procedure.

\smallskip
\begin{enumerate}
\item Apply Algorithm~\ref{algoritme} for computing a basis of $\mathcal{K}_k(G_{k-1},[x_0^T\ 0 \cdots\ 0]^T)$. There we use notation common for Arnoldi iterations: we let
$\underline{\mathcal{H}}_i\in\RR^{(i+1)r\times r i}$
denote the  constructed rectangular block Hessenberg matrix
and $\mathcal{H}_i\in\RR^{r i\times r i}$ the
corresponding $i\times i$ upper blocks.	
\item  A basis for 
\begin{equation}\label{krylovspace}
\mathcal{K}_k(G_N,[x_0^T\ 0 \ \cdots\ 0]^T)
\end{equation}
 is spanned by  the columns of
\begin{equation}\label{defVnk}
V_{N,k}=\left[\begin{array}{cccc} \mathbf{V}_k^T & 0 & \cdots &0  \end{array}\right]^T\in\RR^{ (N+1)n\times k r},
\end{equation}
while, due to the structure of $G_N$, expression
\[
\mathcal{H}_k= V_{N,k}^T\ G_N\ V_{N,k}, 
\]
holds, i.e., $\mathcal{H}_k$ can be considered as an orthogonal projection of $G_N$ on a $k$-dimensional Krylov subspace, for \emph{any} $N$ satisfying (\ref{condN}).
\end{enumerate}

\begin{algorithm}
 \caption{A structure exploiting block Arnoldi algorithm \label{algoritme}}
	\begin{algorithmic}[1]
		\REQUIRE $x_0\in\RR^{n\times r}$ of full column rank, number of iterations $k$
		\STATE Let $x_0= Q_0 \tilde R_0$ be the reduced QR factorization of $x_0$. Set $\mathbf{V}_1=Q_0$ and let $\underline{\mathcal{H}}_0$ be the empty matrix
		\FOR {$i=1,2,\ldots, k$}
		\STATE Let $W_i=G_{i} \left[\begin{array}{c}Q_{i-1}\\0 \end{array}\right]$
		\STATE  Compute $H_i=[\mathbf{V}_i^T\ 0]\ W_i$ 
		and then $\hat{W}_i = W_i - \left[\begin{array}{c}\mathbf{V}_i\\0\end{array}\right] H_i$ (orthogonalization)
		\STATE Compute $\hat W_i= Q_{i} \tilde R_i$ as the reduced QR factorization of $\hat W_i$ (normalization)
		\STATE Let $\underline{\mathcal{H}}_i = \left[\begin{array}{cc}\underline{\mathcal{H}}_{i-1} & H_i \\ 0 & \tilde R_i \end{array}\right] \in\RR^{(i+1)r\times i r}$ 
		\STATE Expand $\mathbf{V}_i$ into $\mathbf{V}_{i+1} =\left[\begin{array}{c|c} \begin{array}{c}\mathbf{V}_i \\ 0 \end{array} & Q_i\end{array}\right]$
		\ENDFOR
		\\
		\hspace*{-0.6cm}\textbf{Output:}\  matrix $\mathbf{V}_{k}$, whose columns are an orthogonal basis for\\
		\hspace*{0.9cm}$\mathcal{K}_k(G_{k-1},[x_0^T\ 0 \cdots\ 0]^T)$, $\mathcal{H}_{k},\ \underline{\mathcal{H}}_k$, satisfying  $\mathcal{H}_{k}=\mathbf{V}_k^T G_{k-1}  \mathbf{V}_k$.
	\end{algorithmic}
\end{algorithm}


\subsection{Dynamic approximation of the transfer function}\label{pardyn2}

We now arrive at the derivation of an approximation of $\Upsilon_N(\slambda)$, defined by (\ref{defUpsilonN}) or, equivalently, (\ref{transferlarge2}), having a prescribed order~$k r$, once again under the condition that (\ref{condN}) is satisfied.
For this we construct the Krylov space 
(\ref{krylovspace})
 and project matrices $F_N, G_N, H_N$, defined in Theorem~\ref{theosparse}, on this Krylov space.
An orthogonal projection yields the following approximation of $\Upsilon_N(\slambda)$:
\begin{equation}\label{approxgammak}
\mathbf{\Upsilon}_{k}(\slambda)=\mathbf{F}_k\ (\slambda \mathbf{G}_k-I)^{-1}\ \mathbf{H}_k,
\end{equation}
where
\begin{equation}\label{defffk}
\begin{array}{lllll}
\mathbf{F}_k&=& C F_{N}\ V_{N,k}&=&C F_{k-1} \mathbf{V}_k, \\
\mathbf{G}_k&=&V_{N,k}^T\ G_N\ V_{N,k}&=&\mathcal{H}_{k}, \\
\mathbf{H}_k&=&V_{N,k}^T\ H_{N}&=&\mathbf{V}_k^T H_{k-1},\ \
\end{array}
\end{equation}
matrix $\mathbf{V}_k$ and $\mathcal{H}_k$ refer to the output of Algorithm~\ref{algoritme}
and  $V_{N,k}$ is given by (\ref{defVnk}).
%
%
The matrices of the reduced model  (\ref{approxgammak}) do \emph{not} depend on $N$. Furthermore, matrices $\mathbf{F}_{k}$ and $\mathbf{H}_{k}$ are \emph{submatrices} of $\mathbf{F}_{k+1}$ and $\mathbf{H}_{k+1}$. Therefore, they can be constructed in a dynamic way when doing iterations of Algorithm~\ref{algoritme},  as is the case with the Hessenberg matrix $\mathbf{G}_k=\mathcal{H}_k$.  

\medskip

With a particular choice of the vector $x_0$ in (\ref{krylovspace}), the transfer function (\ref{approxgammak}) satisfies the following moment matching property with the (original) transfer function (\ref{defUpsilon}) of the time-delay system (\ref{sys}).
\begin{proposition}\cite[Theorem~11]{wimmoment}\label{maintheoreme}
	Let $N,k\in\NN$ with $N\geq k\geq 2$ and let the Krylov space (\ref{krylovspace}) be constructed from
	\[
	x_0=R_0^{-1} B.
	\]
	%
	Then  transfer function
	(\ref{approxgammak})
	satisfies
	\begin{equation}\label{momentprop3bis}
	\left.\frac{d^i \mathbf{\Upsilon}_{k}(\slambda)}{d\slambda^i}\right|_{\slambda=0}
	=\left.\frac{d^i \Upsilon(\slambda)}{d\slambda^i}\right|_{\slambda=0},\ \ i=0,\ldots, k-2
	\end{equation}
	and
	\begin{equation}\label{momenprop4}
	\left.\frac{d^i \mathbf{\Upsilon}_{k}(\slambda^{-1})}{d\slambda^i}\right|_{\slambda=0}
	=\left.\frac{d^i \Upsilon(\slambda^{-1})}{d\slambda^i}\right|_{\slambda=0},\ \ i=0,1.
	\end{equation}
\end{proposition}

Note that Proposition~\ref{maintheoreme} concerns the matching of moments with the transfer function of the  \emph{original} delay system (\ref{sys}). This is due to to the property that the moments, preserved by projection of the discretized system, are precisely matching moments between the discretized system and  the delay system, by Proposition~\ref{theoremmoment}.

\subsection{Dynamic approximation of the delay Lyapunov matrix}\label{pardynlyap}
The evaluation of $\mathcal{P}_N(t)$, defined by (\ref{mainsparse}), relies on solving Lyapunov equation (\ref{lyapunov2}).
An established way to solve large-scale Lyapunov equations consists of computing a low-rank approximation obtained from the projection of the Lyapunov equation on a Krylov space, see, e.g., \cite{simoncini} and the references therein. 

To determine an appropriate Krylov space, it is useful to express $Q_N$ in terms of matrix exponentials,
\begin{equation}\label{defQN-me}
Q_N=\int_0^{\infty} e^{G_N^{-1} s} \left(G_N^{-1} H_N\right)\ \left(H_N^T  G_N^{-T}\right) e^{G_N^{-T}s}ds.
\end{equation}
Hence, a low rank approximation of $Q_N$ can be induced by approximating the action of $e^{G_N^{-1} t}$ on vector(s) $(G_N^{-1} H_N)$ in a low-dimensional space.  This motivates us to include 
\[
G_N^{-1} H_N= \left[\begin{array}{c}R_0^{-1}B \\ 0 \\ \vdots \\ 0\end{array}\right]
\]
 in the Krylov space. Furthermore, since the rightmost characteristic roots of a delay equation are typically very well approximated by the dominant eigenvalues of $G_N$ (equivalently, the smallest eigenvalues of $\mathcal{A}_N$ in modulus), while the largest eigenvalues of $A_N$  have no correspondence with characteristic roots (see the arguments in Section~\ref{parprop} and the illustration in Figure~\ref{figeig}), approximating the dominant eigenspace of $G_N$ should be favored, which brings us once again to Krylov space (\ref{krylovspace}) with starting vector $x_0=R_0^{-1} B$.

 Replacing $Q_N$ in (\ref{lyapunov2}) by $V_{N,k} \mathbf{Q_k} V_{N,k}^T$
%
 and requiring the residual to be orthogonal with respect to the Krylov space, we arrive at the projected Lyapunov equation
\begin{equation}\label{proj-lyap-eqn}
\mathbf{G}_{k} \mathbf{Q}_{k}+ \mathbf{Q}_{k}  \mathbf{G}_{k}^{T}+\mathbf{H}_{k} \mathbf{H}_{k}^T=0.
\end{equation}
Hence, under assumption that $\mathbf{G}_k$ is invertible  we can approximate
\begin{equation}\label{approxQ}
\begin{array}{lll}
Q_N &\approx& V_{N,k}  \mathbf{Q}_{k} V_{N,k}^T
\\
&=& \int_0^{\infty}V_{N,k} e^{ s \mathbf{G}_{k}^{-1}  } (\mathbf{G}_{k}^{-1} \mathbf{H}_{k})\
 (\mathbf{H}_{k}^T \mathbf{G}_{k}^{-T})
e^{ s \mathbf{G}_{k}^{-T} } V_{N,k}^T\ ds.
 \end{array}
\end{equation}
Let us now compare  approximation (\ref{approxQ}) with expression (\ref{defQN-me}). By construction of the Krylov space we have 
\[
G_N^{-1} H_N=V_{N,k} \beta
\]
for some matrix $\beta$ of appropriate dimensions. As a consequence, 
\[
H_N=G_N V_{N,k}\beta \ \Rightarrow \ \mathbf{H}_k= \mathbf{G}_k\beta.
\]
Thus, the approximation of $Q_N$ as in (\ref{approxQ}) can be interpreted in terms of the approximation
\begin{equation}\label{approxexp}
  e^{t G_N^{-1}} (G_N^{-1} H_N)=  e^{t G_N^{-1}} \left(V_{N,k}\right)\beta \approx  V_{N,k} e^{t\mathbf{G}_{k}^{-1}} \beta.
\end{equation}

Substituting the right-hand side of  (\ref{approxQ}) into (\ref{mainsparse}) we get
\begin{equation}\label{comb1}
\begin{array}{ll}
\mathcal{P}_N(t) &\approx F_N V_{N,k} \mathbf{Q}_k V_{N,k}^T  e^{G_N^{-T}t} F_N^T
\\
&= F_N V_{N,k} \mathbf{Q}_{k} \left(e^{t G_N^{-1}} V_{N,k}\right)^T F_N^T.
\end{array}
\end{equation}
To approximate $e^{t G_N^{-1}} V_{N,k}$ we use the same principle underlying (\ref{approxexp}). More precisely,
we build a Krylov space, $\mathrm{span}\left\{V_{N,k},\  G_N V_{N,k},\ \ldots, G_N^{k} V_{N,k} \right\}$. Since the columns of $V_{N,k}$  already span a Krylov space, this can be done by doing $k$ more iterations of Algorithm \ref{algoritme}, provided condition  (\ref{condN}) on $N$ is strengthened to 
\begin{equation}\label{condN2}
2k\leq N.
\end{equation}
It results in a basis $V_{N,2k}$ such that $V_{N,k}=V_{N,2k} \left[\begin{array}{c}I\\ 0\end{array}\right]$, hence,
we can approximate
\begin{equation}\label{comb2}
\left(e^{t G_N^{-1}} V_{N,k}\right) \approx  V_{N,2k} e^{t \mathbf{G}_{2k}^{-1}}  \left[\begin{array}{c}I\\ 0\end{array}\right].
\end{equation}
Finally, combining (\ref{comb1}) and (\ref{comb2})  we arrive at the following  approximation of $\mathcal{P}_N(t)$ and thus of the Lyapunov matrix  $P(t)$,
\begin{equation}\label{approxUk}
\mathbf{P}_k(t)=[R_0\ R_1\  \cdots R_{k-1}] \mathbf{V}_{k} \mathbf{Q}_{k} \left[I\ 0\right]e^{t \mathbf{G}_{2k}^{-T}} \mathbf{V}_{2k}^T
\left[\begin{array}{c} R_0^T\\ R_1^T \\ \vdots \\ R_{2k-1}^T\end{array}\right],
\end{equation}
where $\mathbf{Q}_k$ satisfies (\ref{proj-lyap-eqn}). This brings us to Algorithm~\ref{algoritm2}.

\begin{algorithm}[h]
	\caption{Construction of a (uniformly) low-rank approximation of the the delay Lyapunov matrix \label{algoritm2}}
	\begin{algorithmic}[1]
		\REQUIRE $B\in\RR^{n\times r}$ of full column rank, parameter $k$ determining number of Arnoldi iterations
		\STATE Set $x_0=R_0^{-1} B$ and perform $2k$ iterations of Algorithm~\ref{algoritme}, resulting
		\\
		 in $\mathbf{V}_{2k}$ and $\mathbf{G}_{2k}=\mathcal{H}_{2k}$; set
		 	 \[
		 	   \mathbf{G}_k= \left[\begin{array}{cc}I_{kr} & 0\end{array}\right] \mathbf{G}_{2k} \left[\begin{array}{c}I_{kr}\\ 0\end{array}\right].
		 	   \]
		\STATE Construct matrices $\mathbf{H}_k=\mathbf{V}_k^T H_{k-1}$ and $\mathbf{L}_k=[R_0\ R_1 \cdots R_{2k-1}] \mathbf{V}_{2k}$. 
		\STATE Solve Lyapunov equation (\ref{proj-lyap-eqn}) for $\mathbf{Q}_k$. \\
		\hspace*{-0.6cm}\textbf{Output:}\  matrices $\mathbf{L}_k,\ \mathbf{Q}_k,\  \mathbf{G}_{2k}$ from which $\mathbf{P}_k$
		can be constructed \\
		\hspace*{0.8cm} according to (\ref{approxUk}).
	\end{algorithmic}
\end{algorithm}

\medskip

Finally we note that the low-order approximation (\ref{approxgammak}) of transfer function $\Upsilon$ and the approximation (\ref{approxUk}) of Lyapunov matrix $P(t)$ of rank smaller or equal to $kr$ are still consistent, in view of Proposition~\ref{prop-cor1} and Proposition~\ref{prop-cor2}. 
\begin{proposition}\label{prop-cor3}  We can express
	$
	\left\|\mathbf{\Upsilon}_{k}\right\|_2^2=\mathrm{Tr}\left(C \mathbf{P}_k(0) C^T
	\right)
	$
\end{proposition}
\begin{proof} From (\ref{approxUk}) we directly have
	\[
	\begin{array}{lll} 
	\mathrm{Tr}\left(	C \mathbf{P}_k(0) C^T\right) &=&  \mathrm{Tr}\left( C F_N \mathbf{V}_{N,k} \mathbf{Q}_k  \mathbf{V}_{N,k}^T F_N^T C^T\right)
	\\
	&=& \mathrm{Tr}\left(\mathbf{F}_k \mathbf{Q}_k  \mathbf{F}_k^T\right)
	\end{array}
	\]
	The latter expression, combined with (\ref{proj-lyap-eqn}), characterize the $\mathcal{H}_2$ norm of $\mathbf{\Upsilon}_k$.	
\end{proof}

\subsection{Implementation aspects and  computational complexity} \label{parimpl}

Algorithm~\ref{algoritm2} is fully dynamic, in the sense that by increasing iteration count  $k$,  matrices $\mathbf{V}_k$, $\mathbf{G}_k$,  $\mathbf{L}_k$, etc., only need to be extended or updated, hence, the iteration can be resumed if the accuracy is deemed insufficient. If $k$ is not chosen a-priori, this brings us to discuss stopping criteria.
The most reliable approach consists of testing the residual for  boundary value problem (\ref{BVP}) at a set of time-instants in the interval under consideration. Substituting (\ref{approxUk}) in (\ref{BVP}) and letting the columns of $\mathcal{W}_k$ be on orthogonal basis for the column space of
$
\left[\mathbf{L}_{k}\ A_0 \mathbf{L}_{k}\ \cdots\ A_m\mathbf{L}_{k}  \right],
$
every term in the equations has its column, respectively row range contained in those of $\mathcal{W}_k$, respectively $\mathcal{W}_k^T$. As a consequence the Euclidean norm of the residual at a given time-instant can be expressed in terms of the residual for a boundary value problem where the size of the matrices is determined by the rank of $\mathcal{W}_k$. 

The construction of matrix $\mathcal{W}_k$ however introduces a significant additional computational cost.  To our experience a good indicator of convergence consists of determining the residual for Lyapunov equation (\ref{lyapunov2}).  Recall that $G_N Q_N $ is approximated by
\[
G_N V_{N,k} \mathbf{Q}_k V_{N,k}^T=V_{N,k+1}\underline{\mathcal{H}}_k \left[\mathbf{Q}_k\ 0\right]  V_{N,k+1}^T.
\]
At the same time we have 
\[
H_N=V_{N,k}V_{N,k}^T H_N=V_{N,k} \mathbf{H}_k= V_{N,k+1} \left[ \begin{array}{c} \mathbf{H}_k \\ 0 \end{array}\right].
\]
Since the columns of $V_{N,k+1}$ are orthogonal, the residual of (\ref{lyapunov2}), $R_{N,k}$ satisfies
{\small
	\begin{equation}\label{residualnorm}
	\|R_{N,k}\|_2=\left\|
	\underline{\mathcal{H}}_k \left[\mathbf{Q}_k\ 0\right] 
	+\left[\begin{array}{c}\mathbf{Q}_k^T\\ 0\end{array}\right]\underline{\mathcal{H}}_k^T + 
	\left[\begin{array}{c} \mathbf{H}_k \\0 \end{array}\right]  
	\left[\begin{array}{ccc}  \mathbf{H}_k^T  &0 \end{array}\right]
	\right\|_2.
	\end{equation}}
Note that the residual norm can be expressed in terms of projected matrices and is independent of $N$.

What concerns the computational complexity, the core of Algorithm~\ref{algoritm2} consists of doing $2k$ iterations of Algorithm~\ref{algoritme}. Expressed in terms of operations on \emph{vectors of length $n$}, the computational complexity is as follows:
\begin{center}
	\begin{tabular}{ll}  
		number of backward solves: & $2rk$,   \\
		number of matrix vector products: & $O(r k^2)$,     \\
		number of scalar products (orthogonalization):   & $O(r^2 k^3)$. \\
	\end{tabular}
\end{center}
It is important to point out that all backwards solves are with the same matrix ($R_0$), inherent to an Arnoldi type algorithm. Hence, the first step in our implementation consists of computing a (sparse) LU factorization of matrix $R_0=\sum_{i=0}^m A_i$.  For the remaining steps of Algoirthm~\ref{algoritm2} the dominant cost in most cases consists of solving Lyapunov equation 
(\ref{proj-lyap-eqn}) for $\mathbf{Q}_k$, whose complexity is described by $\mathcal{O}(r^3 k^3)$ operations for the adopted Bartels-Stewart algorithm. 
In addition, our implementation fully exploits the property that, due to the special structure of $G_k$ and the starting vector of the Arnoldi iteration,  $\mathbf{V}_k$ can be represented in the form
\begin{equation}\label{repsparse}
\mathbf{V}_k= \left( I_k\kron W_k\right)
\left[\begin{array} {llll} 
v_{1,1}  & v_{1,2} & \cdots & v_{1,k} \\
0 &  v_{2,2}  &  \dots & v_{2,k} \\
\vdots &  \ddots  & \ddots        & \vdots \\
0  & \cdots           &   0     & v_{k,k}
\end{array}\right],
\end{equation}
where both factors are orthogonal matrices, matrix $W_k$ has dimensions $n\times s$ with $s\leq kr$  and $v_{i,j}\in\mathbb{R}^{s\times r},\ i,j=1,\ldots,k$. Furthermore, both factors can be dynamically constructed. These properties are fundamental in the so-called tensor infinite Arnoldi method and CORK framework (COmpact Rational Krylov algorithms) for nonlinear eigenvalue problems \cite{cork,jarlebring-tensor}, on their turn generalizing \cite{Bai:2005:SOAR} for quadratic eigenvalue problems. We refer to these references for more details on representation~(\ref{repsparse}). Obviously, for large $n$ its use leads to a significant reduction in the memory requirements, but it is also beneficial in terms of computational complexity, as argued in \cite{jarlebring-tensor}.

\subsection{Interpretation in terms of projections of an infinite-dimensional system}\label{secinfcheb}

The spectral discretization in Section~\ref{secfinal}, resulting in a finite-dimensional approximation of dimension $(N+1)n$, played a major role in the technical derivation of Algorithm~\ref{algoritm2}. However, eventually the role of parameter $N$ is marginal:
\begin{itemize}
	\item the execution of Algorithm~\ref{algoritm2} (and Algorithm~\ref{algoritme} on which it relies), as well as the discussed stopping criteria, do not rely on a choice of $N$;
	\item the algorithms are dynamic in the sense that the iterative processes can always be resumed;
	\item Proposition~\ref{maintheoreme} connects moments of transfer functions $\mathbf{\Upsilon}_k$ and $\Upsilon$ directly.
	\end{itemize}
As a matter of fact it is only implicitly assumed  that $N$ is \emph{sufficiently large} (such that (\ref{condN2}) holds). A limit argument, for $N\rightarrow\infty$, provides some intuition for the existence of an interpretation of Algorithm~\ref{algoritm2}  as an algorithm acting on an infinite-dimensional linear system equivalent to (\ref{sys}).  This is also suggested by Figure~\ref{rankcond}, where the singular value functions of $Q_N$ uniformly converge on compact intervals to the limit function displayed in black color. In what follows we make a connection with an infinite-dimensional linear system  concrete.

We reconsider system (\ref{sys}) and define 
\[
v(\theta,t)=x(t+\theta),\ \  \theta\in[-\tau_m,\ 0],\ \ t\geq 0.
\]
Solutions of (\ref{sys}), starting at $t=0$, are continuous for $t\geq 0$, and they satisfy the advection PDE
\begin{equation}\label{pde-inf}
\left\{\begin{array}{ll}
\frac{\partial v}{\partial t} (\theta,t)-\frac{\partial v}{\partial \theta} (\theta,t)=0, & \theta\in [-\tau_m,\ 0),\ \ t\geq \tau_m,\\
\frac{\partial v}{\partial t}(0,t)=A_0 v(0,t)+\sum_{i=1}^m A_i v(-\tau_i,t)+ B u(t),\ & t\geq \tau_m,\\
\end{array}\right.
\end{equation}
see \cite{krstic}. Let us represent $v(\theta,t)$ in a Chebyshev series in variable $\theta$ on the interval $[-\tau_m,\ 0]$,
\[
\begin{array}{l}
v(\theta,t)=\sum_{j=0}^\infty c_j(t) T_j\left( \frac{2\theta}{\tau_m}+1\right),\ \theta\in [-\tau_m,0]. 
\end{array}
\]

The second equation in (\ref{pde-inf}) then becomes
\begin{equation}\label{firstseries}
\begin{array}{lll}
\sum_{j=0}^{\infty} \dot c_j(t)&=&A_0 \left(\sum_{j=0}^\infty c_j(t)\right)  +\sum_{i=1}^m A_i \left(\sum_{j=0}^\infty c_j(t) T_j\left( -\frac{2\tau_i}{\tau_m}+1\right)\right) 
\\
&=&\sum_{j=0}^{\infty} c_j(t) \left(A_0+\sum_{i=1}^m A_i T_j\left( -\frac{2\tau_i}{\tau_m}+1\right) \right).
\end{array}
\end{equation}
In the same way the first equation in (\ref{pde-inf}) becomes
\begin{equation}\label{secondseries}
\sum_{j=0}^{\infty} \dot c_j(t) T_j\left( \frac{2\theta}{\tau_m}+1\right)=\sum_{j=1}^{\infty} c_j(t)  \frac{2j}{\tau_m} U_{j-1} \left( \frac{2\theta}{\tau_m}+1\right),
\end{equation}
where we employed the property $$\dot T_{j+1}(\theta)=(j+1) U_j(\theta),$$ with $U_j$ the Chebyshev polynomial of the second kind and order $j$, for $j\geq 0$.

For $j\geq 2$ we can substitute expression
\[
T_j\left( \frac{2\theta}{\tau_m}+1\right)=\frac{1}{2} \left( U_j\left( \frac{2\theta}{\tau_m}+1\right)- U_{j-2}\left( \frac{2\theta}{\tau_m}+1\right)\right)
\]
in (\ref{secondseries}), as well as 
\[
T_1\left( \frac{2\theta}{\tau_m}+1\right)=\frac{1}{2}  U_1\left( \frac{2\theta}{\tau_m}+1\right),\ \ T_0\left( \frac{2\theta}{\tau_m}+1\right)=U_0\left( \frac{2\theta}{\tau_m}+1\right).
\] 
Multiplying subsequently left and right hand side of (\ref{secondseries}) with
$$
U_{i-1} \left( \frac{2\theta}{\tau_m}+1\right) 
\sqrt{1-\left( \frac{2\theta}{\tau_m}+1\right)^2},
$$
taking the integral in $\theta$ from $-\tau_m$ to zero, and considering the orthogonality properties of Chebyshev polynomials of the second kind, we arrive at
\begin{equation}\label{thirdseries}
\begin{array}{l}
\dot c_0(t)-\frac{1}{2}\dot c_2(t)	=\frac{2}{\tau_m} c_1,
\\
\frac{1}{2} \left( \dot c_{i-1}(t)-\dot c_{i+1}(t)  \right)	=\frac{2i}{\tau_m} c_i,\ \   i\geq 2.
\end{array}
\end{equation}
Letting $\mathbf{c}=\left[c_0^T\ c_1^T\ \cdots \right]^T$, $\mathbf{e_1}=[1 \ 0\ \cdots]^T$ and $\mathbf{1}=[1\ 1\ \cdots ]^T$,  differential equations (\ref{secondseries}) and (\ref{thirdseries}) can be written as
\begin{equation} \label{infseries}
\left\{
\begin{array}{lll}
\Pi_\infty \dot{\mathbf{c}}(t) &=&\Sigma_\infty \mathbf{c}(t) +\left(\mathbf{e_1}\otimes B\right) u(t), \\
y(t)&=&\left(\mathbf{1}^T \otimes C\right)  \mathbf{c}(t),
\end{array}\right.
\end{equation}
with 
{\small
	\begin{equation}\label{defPiN2-2}
	\Pi_\infty=\frac{\tau_{m}}{4}\left[\begin{array}{rrrrrr}
	\frac{4}{\tau_m} &\frac{4}{\tau_m} & \frac{4}{\tau_m}&\cdots&\cdots&\cdots \\
	2& 0&-1 &&& \\
	&\frac{1}{2} &0 &-\frac{1}{2}&&\\
	&&\frac{1}{3}&0&-\frac{1}{3}&\\
	&&&\ddots &\ddots&\ddots\\
	\end{array}\right]\otimes I
	\end{equation}
}
and
\begin{equation}\label{defSigmaN2-2}
\Sigma_\infty=\left[\begin{array}{cccc}
R_0 & R_1& \cdots&    \\
& I_n &   & \\
&   & \ddots & 
\end{array}
\right].
\end{equation}

System (\ref{infseries})-(\ref{defSigmaN2-2}) can be interpreted as alternative  representation of (\ref{pde-inf}), and of the original delay equation (\ref{sys}).
At the same time, system (\ref{approx-extra}), obtained after a spectral discretization and at the basis of the approach spelled out in the previous sections, is equivalent to 
\begin{equation}\label{approx3}
\left\{\begin{array}{rcl}
\Pi_N\dot c_N(t)&=&\Sigma_N c_N(t)+(\mathbf{1}_N\otimes B) u(t),\\
y(t)&=&(\mathbf{1}_N^T\otimes C) x(t)
\end{array}\right.
\end{equation}
 since
\[
G_N\Sigma_N^{-1} (\mathbf{1}_N\otimes B)=H_N,\ \ (\mathbf{1}_N^T\otimes C) G_N^{-1}=F_N.
\]

System (\ref{approx3}) can be  obtained from (\ref{infseries}) by truncating the state to the first $N+1$ block components (or, equivalently applying a Galerkin projection on the range of $\left( [I_{n(N+1)} \ 0 \cdots\ 0]^T\right)$.
Algorithms~\ref{algoritme}-\ref{algoritm2} only rely on the use of submatrices of $\Sigma_N,\Pi_N$, at top-left position (recall definition (\ref{defGGN}) of $G_N$), which are on their turn ``submatrices" of $\Pi_{\infty}$ and $\Sigma_\infty$.  Therefore, these algorithms can be interpreted as applied to  infinite-dimensional system (\ref{infseries}) directly. 
We note that, for case of approximating characteristic roots by the reciprocal of eigenvalues of $\mathbf{G}_k$, a related interpretation of Algorithm~\ref{algoritme} is given  in \cite{jarlebring-sisc}, in terms of an operator eigenvalue problem.

Finally, an overview of the developments throughout the Sections~\ref{secfinal}-\ref{secproject} is given by Figure~\ref{figoverview}
\begin{figure}
\begin{center}
\resizebox{13.5cm}{!}
{

\tikzstyle{block} = [rectangle, text width=4cm, text badly centered]
\tikzstyle{block2} = [rectangle, text width=2cm, text badly centered]

\tikzstyle{blockline} = [draw, dashed, fill=white, rectangle, text width=10cm, text badly centered]
\tikzstyle{blockline2} = [draw, dashed, fill=white, rectangle, text width=3.5cm, text badly centered]
\begin{tikzpicture}[>=stealth'] 

\node [block] (system) {System};
\node [block, right=0.1cm of system] (approx) {Delay Lyapunov matrix \\ (approximation)}; 
\node [block, right=0.1cm of approx] (standard) { Variable of standard Lyapunov matrix equation};

\node [block,text width=3cm, below=0.3cm of system] (systemeq) {$A_0$,$\ldots$,$A_m$,$B$,$C$\\ $\tau_1$,$\ldots$,$\tau_m$};
\node [block, right=0.6cm of systemeq] (approxeq) {$P(t)$};
\node [block,text width=3cm, right=1.1cm of approxeq] (xeq) {};

\node [block, below=1.3cm of systemeq] (seq) {$\mathcal{A}_N$,$B_N$,$C_N$\\ $\sim$$C F_N$,$G_N$,$H_N$\\\footnotesize{$\dim=(N+1)n$}};
\node [block2, right=1.1cm of seq] (aeq) {$\mathcal{P}_N(t)$};
\node [block2, right=2.1cm of aeq] (st) {$\begin{array}{c}P_N\\ Q_N\end{array}$};

\draw [->, thick] (st) edge[dashed] (aeq);

\node [block, below=1.2cm of seq] (fseq) {$\mathbf{F}_k$,$\mathbf{G}_k$,$\mathbf{H}_k$\\\footnotesize{$\dim={kr}$}};
\node [block2, right=1.1cm of fseq] (faeq) {$\mathbf{P}_k(t)$\\\footnotesize{$\text{rank}\leq{2kr}$}};
\node [block2, right=2.1cm of faeq] (fst) {$\mathbf{Q}_k$\\\footnotesize{$\dim={kr}$}};

\draw [->, thick] (fst) edge[dashed] (faeq);
\draw [->,thick] (systemeq) edge (seq);

\node [blockline2, below=0.4cm of systemeq] (spdi) { \textit{Spectral Discretization}};

\draw [->,thick] (seq) edge (fseq);
\draw [->,thick] (aeq) edge (faeq);
\draw [->,thick] (st) edge (fst);

\node [blockline, below=0.6cm of aeq] (prk) { \textit{Projection on Krylov Space} ($2k\leq N$)};

\node[fit= (spdi) (seq) (aeq) (st) (prk), 
draw, inner sep=0.15cm, fill opacity=0.1,fill=black!70!white]   (Box)           {};


\node[below left= 0.5cm and -1cm of system] (point1) {};
\node[right=10.2cm of point1] (point2) {};

\node [block, left=0cm of systemeq] (equiv) {$\sim\Pi_{\infty}$,$\Sigma_{\infty}$,\\
	$(\mathbf{e}_1\otimes B)$, $(\mathbf{1}^{T}\otimes C)$};

\draw [->,thick] (point1) edge[bend right=70] node [left,text width=2cm,text badly centered] {	Projection} (fseq);
\draw [->,thick] (point2) edge[bend left=70] (fst);

\end{tikzpicture}

} 
\end{center}
\caption{Overview of different steps in the derivations and corresponding notations. If only the $\mathcal{H}_2$ norm needs to be approximated, a Kyrlov space of dimension $k$ such that $k\leq N$ is sufficient. 
With the relation between $k$ and $N$ satisfied, the delay Lyapunov matrix and $\mathcal{H}_2$ norm approximations, obtained after projection of the discretized system, do not depend on the value of $N$, only on $k$. This leads to the interpretation spelled out in Section~\ref{secinfcheb} and illustrated with the curves arrows.	
	\label{figoverview}}
\end{figure}
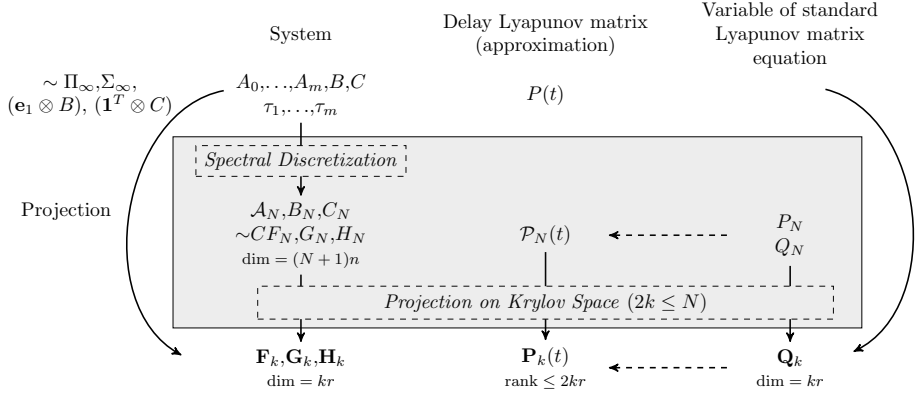

\section{Experiments}\label{secapplic}


We first consider the model for a heat exchanger described in~\cite{vyhlidal}, for which the controller (based on a combination of static state feedback and  proportional integral (PI) control) has been determined by optimizing the spectral abscissa using the method of~\cite{wimiet}. The closed-loop system is described by a delay equation of form (\ref{sys}) with $n=5$ state variables and $m=7$ delays. The non-zero elements of matrices $A_i,\ i=0,\ldots,7$, are specified as in the following table,
\begin{equation}\label{ex11}
{\small
\begin{array}{l|l}
A_0&\ \ \   (2,1):\ \frac{1}{3},\ (2,2):\ -\frac{2}{3},\ (3,3):\ -\frac{1}{3}\ (5,4):\ -1
\\
A_1&\ \ \  (4,3):\  0.0324
\\
A_2&\ \ \   (1,1):\ -0.07142857143
\\
A_3&\ \ \   (4,4):\   -0.04
\\
A_4&\ \ \   (2,4):\   \frac{1}{3}
\\
A_5&\ \ \  (1,1):\   -0.01219364644,\ (1,2):\ -0.05460277319,\ (1,3):\ -0.1005215423 \\
& \ \  \ (1,4):\ -0.1290047174,\ (1,5):\ 0.005063395489
\\
A_6&\ \ \   (3,2):\  0.3133333333
\\
A_7&\ \ \    (1,2):\      0.01714285714,
\end{array}
}
\end{equation}
while input matrix $B$, output matrix $C$ and the delay values are given by
\begin{equation}\label{ex12}
{\small
B=
\left[\begin{array}{c} 
0.0278571429
\\
0
\\
0
\\
0
\\
0
\end{array}\right],\ \ C=I,\ 
\left[\begin{array}{c}
\tau_1\\
\tau_2\\
\tau_3\\
\tau_4\\
\tau_5\\
\tau_6\\
\tau_7
\end{array}\right]=
\left[\begin{array}{c}
 2.8\\
6.5\\
9.2\\
13\\
13.2\\
18\\
40
\end{array}\right].
}
\end{equation}

In Figure~\ref{figheater2} we plot the normalized error on the Lyapunov matrix, 
\begin{equation}\label{figer1e}
\frac{\max_{t\in[0,\ t_{\max}]} |P(t)-\mathbf{P}_k(t)|}{\max_{t\in[0,\ t_{\max}]} |P(t)|}, 
\end{equation}
for  $t_{\max}=50$ as a function of $k$, computed using Algorithm~\ref{algoritm2}. We also show the normalized error on the $\mathcal{H}_2$ norm,
\begin{equation}\label{figer2e}
\frac{\left|\|\Upsilon\|_2-\|\mathbf{\Upsilon}_k\|_2\right|}{\|\Upsilon\|_2}.
\end{equation}
Finally, the evolution of selected elements of the Lyapunov matrix $P(t)$ is shown in Figure~\ref{simul}.

Even though the dimension $n$ is small, the advantage of using a projection method is significant. To illustrate this, when choosing $k=100$ the application of Algorithm~\ref{algoritm2} involves the solution of a matrix Lyapunov equation of dimension $100\times 100$, leading to an error on the $\mathcal{H}_2$ norm approximation smaller than $2\ 10^{-8}$, see Figure~\ref{figheater2}. At the same time, when discritizing the delay equation into an ordinary equation as in Section~2, with $N=19$, and computing an $\mathcal{H}_2$ norm approximation via (\ref{approxH2}), one also has to solve a Lyapunov equation of size $100\times 100$, but the error is then around $10^{-6}$. The underlying reason is that the former approach can be interpreted in terms of a much more accurate discretization with  $N>99$ points, followed by $100$ steps of an Arnoldi iteration (see Figure~\ref{figoverview}).

\begin{figure}
	\begin{center}
		\resizebox{!}{8cm}{\includegraphics{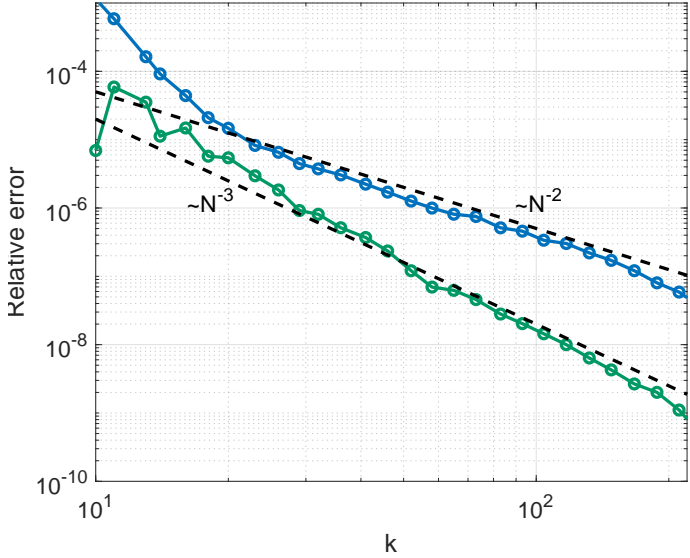}}
	\end{center}
	\caption{\label{figheater2} Normalized errors (\ref{figer1e}) with $t_{\max}=50$ (blue curve) and (\ref{figer2e}) (green curve) as a function of $k$ for system (\ref{sys}) with matrices and delays (\ref{ex11})-(\ref{ex12}). The dashed lines indicate the rates $\mathcal{O}\left(k^{-2}\right)$ and $\mathcal{O}\left(k^{-3}\right)$. }
\end{figure}

\begin{figure}
	\begin{center}
		\resizebox{!}{6cm}{\includegraphics{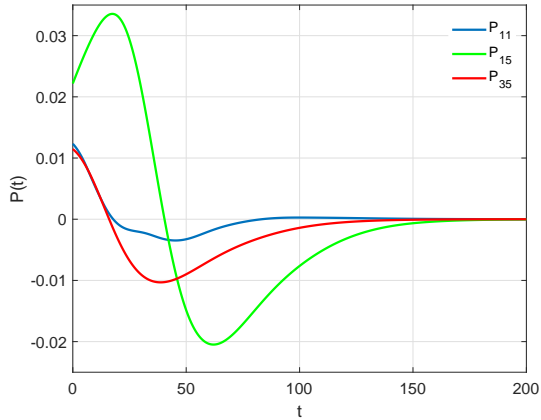}}
	\end{center}
	\caption{\label{simul} Some elements of $P(t)$ as a function of $t$ for system (\ref{sys}) with matrices and delays 
		(\ref{ex11})-(\ref{ex12}).}
\end{figure}

For the second and third example we  consider   models described
by partial differential equations (PDE)
\begin{equation} \label{pde1}
\frac{\partial v(x,t)}{\partial t}=
\frac{\partial^2 v(x,t)}{\partial
  x^2}-\frac{1}{4}x\ v(x,t-1)
\end{equation}
and
\begin{equation} \label{pde2}
\frac{\partial v(x,t)}{\partial t}=
\frac{\partial^2 v(x,t)}{\partial
	x^2}-2\sin(x)v(x,t)+2\sin(x)v(\pi-x,t-1),
\end{equation}
with in both cases $v(0,t)=v(\pi,t)=0$. The equations, which are variants of
examples in~\cite{Breda:2009:PDE},  can be interpreted as
heat equations describing in the temperature in a rod, controlled with distributed delayed feedback. In (\ref{pde1}) the feedback is proportional and localized, in (\ref{pde2}) it of Pyragas type and non-localized.
We discretize differential equations (\ref{pde1})-(\ref{pde2}) in space using central differences. For (\ref{pde2}), for instance, this resulting a systems of the form (\ref{sys-z}) with matrices
\[
A_0=\left(\frac{n-1}{\pi} \right)^2
\left[\begin{array}{ccccc}
-2 & 1 &  & & \\ 
1 & -2 & 1 & & \\
&\ddots&\ddots &\ddots &  \\
& &1 & -2 & 1\\
&&&1 &-2
\end{array}\right] -2\Delta_0
\]
and $A_1=2\Delta_{-1}$. Here $A_0,\ A_1\in\mathbb{R}^{n\times n}$, and $\Delta_0$ is a diagonal matrix containing the elements of the vector 
$\left(0,\sin\left(\frac{1}{n-1}\pi\right),\cdots, \sin\left(\frac{n-2}{n-1}\pi\right),0\right)$ 
on its diagonal, while $\Delta_{-1}$ is the anti-diagonal vector based on the same vector. 
For both (\ref{pde1}) and (\ref{pde2})  we we take $n=10000$ 
and output matrix $C=(1,1,\ldots,1)/\|(1,1\ldots,1)\|_2$, i.e.,
the output is the average temperature of the rod. We further assume $B=C^T$.

In Figure~\ref{figpde}  we display the normalized error (\ref{figer1e}) on the Lyapunov matrix for the interval $[0,\ t_{\max}]=[0,\ 3],$ as well as the normalized error on the associated $\mathcal{H}_2$ norm approximation, as a function of $k$. 
To shed a light on the computation time, for system (\ref{pde2}) and $k=100$ the computation time for the delay Lyapunov matrix, respectively $\mathcal{H}_2$ norm\footnote{As can be seen from (\ref{approxUk}) only $\mathbf{V}_k$  needs to be available to evaluate $\mathbf{P}_k(t)$ at $t=0$.},  was $42$ seconds, respectively $4.8$ seconds, using MATLAB R2017b on a laptop with Intel Core i7 2.80 GHz processor and 16GB RAM.

\begin{figure}
	\begin{center}
		\resizebox{!}{9cm}{\includegraphics{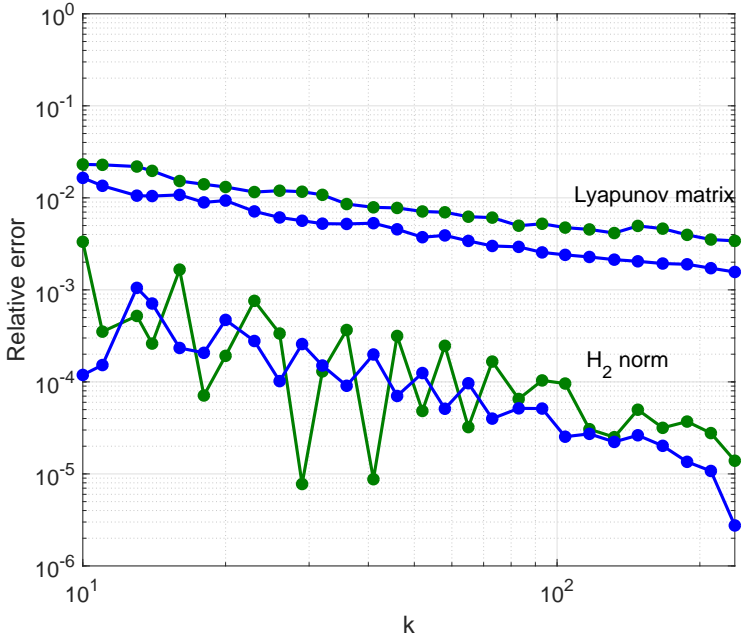}}
	\end{center}
	\caption{\label{figpde}
		 Normalized errors (\ref{figer1e}), with $t_{\max}=3$, and (\ref{figer2e}) as a function of $k$, with matrices obtained form the spatial discretization of (\ref{pde1}) (blue curves) and (\ref{pde2}) (green curves), such that $n=10000$.  }
\end{figure}

Let us now comment on the convergence behavior shown in Figures~\ref{figheater2} and \ref{figer1e}. 
The  experiments carried out for
\[
k r>>n,
\]
 which is natural if $n$ is small as for the first presented example, indicate an asymptotic rate of $\mathcal{O}(k^3)$, respectively $\mathcal{O}(k^2)$  for the $\mathcal{H}_2$ norm, respectively the delay Lyapunov matrix approximation.  These rates are similar to those obtained by the spectral discretization in Section~2 (as a function on $N$), hence, the projection step does not result in a slowing down of the asymptotic convergence rate (recall the arguments in Section~\ref{parprop} where the rates are, among others, related to the lack of smoothness of $P(\cdot)$), even though it is highly advantageous from the point of view of computational complexity. Some intuition behind this observation is given by Theorems~\ref{theoremmoment}  and \ref{maintheoreme}: by construction precisely the matching moment between $\Upsilon$ and $\Upsilon_N$ carry over to the projected transfer function $\mathbf{\Upsilon}_k$.  
 In experiments with very large $n$, we have  $kr << n$ for a realistic range for $k$ values as in the second and third example,  and the observed decay rate is slower,
 which is illustrated by a comparison between Figure~\ref{figpde}
and Figure~\ref{figheater2}.  A possible explanation is that unlike the previous case a low-rank approximation of Lyapunov matrix $P(t)\in\RR^{n\times n}$ is enforced by construction.

Inherent to the projection approach, the efficiency of the computational approach depends on whether or not accurate low rank approximations exist, whose determining factors are not well understood, and the projected system matrix $\mathbf{G}_{2k}$ must be stability preserving (this is the case for most problems  and it was an important consideration in the methodological choices, but not always - a counter example  is the 2nd example in \cite{jarpol} for $n=1023$, where spurious roots are observed in the right half plane). The latter is not necessarily a strong limitation for the $\mathcal{H}_2$ norm computation, since the $\mathcal{L}_2$ norm of the low-order, projected transfer function $\mathbf{\Gamma}_k$ can still be computed using other techniques different from solving the Lyapunov equation directly.  All these issues, and related fixes are subject for further investigation.

\section{Concluding remarks}\label{secconcl}

A novel algorithm for computing delay Lyaopunov matrices and $\mathcal{H}_2$ norms has been presented, which is the first algorithm generally applicable to linear time-delay systems with multiple delays and at the same time having favorable scaling properties with respect to dimension $n$ (the examples  with $n=10000$ in Section~\ref{secapplic} indicates the potential of the approach).  Furthermore, the algorithm is dynamic in nature, in the sense that the computations can be resumed if the accuracy is judged insufficient.  The algorithm results in approximations of the delay Lyapunov matrix in an {explicit form} given by~(\ref{approxUk}). 

Computing delay Lyapunov matrices induces a lot of challenges and complication compared to solving classical Lyapunov matrix equations (making the leap from an algebraic equation to  matrix valued boundary problem (\ref{BVP}) with a non-smooth solution).  At the same time the research  is in still an initial  phase, with to the best of our knowledge, for the moment only two methods available applicable to large problems, the presented one and the one of \cite{jarpol}, which are fundamentally different.  Therefore we hope that the methodology, results and observations trigger further research on the topic.

Finally we come back to the assumption of exponential stability of (\ref{sys}). It implies that computing the Lyapunov matrix (when alternatively defined as the solution of (\ref{BVP}) and not via the fundamental solution), with the presented method is not useful in the context of verifying recent stability conditions, precisely expressed precisely in terms of the delay Lyapunov matrix (see, e.g., \cite{cuvas} and the references therein). Yet, the overall algorithm starts with iterations of  Algorithm~\ref{algoritme}, which corresponds to the Infinite-Arnoldi algorithm \cite{jarlebring-sisc} for eigenvalue computations and which does require an exponentially stable system. Consequently, from the output of the first step, more precisely from the spectrum of $\mathbf{G}_{2k}$, we directly obtain a certificate whether or not the system is exponentially stable.

\section*{Acknowledgements}

The first author thanks V.L. Kharitonov for an invitation to give a talk in a session on Lyapunov matrices at the 14th IFAC Workshop on Time-Delay System, which was the starting point of this work.
The research was supported by the project C14/17/072 of the KU Leuven Research Council, by the project G0A5317N of the Research Foundation-Flanders (FWO - Vlaanderen), and by the project UCoCoS, funded by the European Unions Horizon 2020 research and innovation programme under the
Marie Sklodowska-Curie Grant Agreement No 675080.

\bibliographystyle{plain}

\bibliography{common-VF,wim-VF,referentielijst}

\end{document}